\documentclass[preprint,sort&compress]{elsarticle}
\journal{Applied Numerical Mathematics}

\usepackage{amsmath}
\usepackage{amsfonts}
\usepackage{amsthm}
\usepackage{bbm}
\usepackage{hyperref}
\usepackage[capitalise]{cleveref}
\usepackage{empheq}
\usepackage{forest}
\usepackage{lineno}
\usepackage{mathtools}
\usepackage{mleftright}
\usepackage{pdflscape}
\usepackage{siunitx}
\usepackage{subcaption}

\crefname{equation}{}{}
\Crefname{equation}{Equation}{Equations}

\newtheorem{theorem}{Theorem}
\newtheorem{remark}{Remark}

\newtheorem{corollary}{Corollary}

\newenvironment{butchertableau}[2][1.15]{\def\arraystretch{#1}\arraycolsep=2.75pt\array{#2}}{\endarray}

\newcommand{\abs}[1]{\left\lvert#1\right\rvert}
\newcommand{\norm}[1]{\left\lVert#1\right\rVert}

\DeclareMathOperator{\diag}{diag}
\DeclareMathOperator{\col}{col}

\newcommand{\R}{\mathbb{R}}
\newcommand{\C}{\mathbb{C}}

\renewcommand{\Re}{\operatorname{Re}}
\newcommand{\ud}[1]{\, \mathrm{d}#1}
\newcommand{\nvar}{N}
\newcommand{\timevar}{t}
\newcommand{\ts}{\timevar_{0}}
\newcommand{\tf}{\timevar_{\mathrm{f}}}
\newcommand{\dt}{h}
\newcommand{\one}{\mathbbm{1}}
\newcommand{\order}[1]{\ensuremath{\mathcal{O}\mleft(#1\mright)}}
\newcommand{\psl}{p_{\text{SL}}}
\newcommand{\so}[1]{\gamma_{#1}}
\newcommand{\qo}[1]{\widehat{\gamma}_{#1}}

\begin{document}

\begin{frontmatter}

\title{Adaptive Diagonally Implicit Runge--Kutta Methods Devoid of Order Reduction for Semilinear ODEs}

\author{Steven B. Roberts\corref{cor}}
\ead{roberts115@llnl.gov}
\cortext[cor]{Corresponding Author}
\affiliation{
    organization={Lawrence Livermore National Laboratory},
    addressline={7000 East Ave},
    city={Livermore},
    postcode={94550},
    state={California},
    country={United States}
}

\author{Abhijit Biswas}
\ead{abhijit@iitk.ac.in}
\affiliation{
    organization={Indian Institute of Technology Kanpur},
    addressline={Kalyanpur},
    city={Kanpur},
    postcode={208016},
    state={Uttar Pradesh},
    country={India}
}

\author{David Shirokoff}
\ead{david.g.shirokoff@njit.edu}
\affiliation{
    organization={New Jersey Institute of Technology},
    addressline={154 Summit Street},
    city={Newark},
    postcode={07102},
    state={New Jersey},
    country={United States}
}

\author{Benjamin Seibold}
\ead{seibold@temple.edu}
\affiliation{
    organization={Temple University},
    addressline={1805 N. Broad Street},
    city={Philadelphia},
    postcode={19122},
    state={Pennsylvania},
    country={United States}
}

\begin{abstract}
Diagonally implicit Runge--Kutta (DIRK) methods are a prominent class of numerical methods for solving stiff systems of ordinary differential equations (ODEs). Stiffness does not only impose stability challenges on Runge--Kutta methods; it can also degrade the order of convergence. This so-called order reduction phenomenon occurs when assumptions used for classical convergence analysis, e.g., an asymptotically small step size, fail to hold. In a prior paper by the authors, sharp order conditions and global error bounds for Runge--Kutta methods were developed, which hold uniformly with respect to stiffness when applied to a wide class of semilinear ODEs. In this work, those conditions are leveraged to construct the first DIRK methods of order four and five which satisfy these conditions and thus do not exhibit order reduction. Numerical results demonstrate that for a broad class of relevant nonlinear test problems, these new methods successfully mitigate order reduction, accurately estimate local error via an embedding for adaptive step size control, and can outperform classical DIRK methods.
\end{abstract}

%


\end{frontmatter}

\section{Introduction}

This work focuses on implicit Runge--Kutta methods applied to stiff semilinear differential equations of the form
\begin{equation} \label{eq:ODE}
    y' = f(y) \coloneqq Jy + g(y),  
    \qquad
    y(\ts) = y_0,
    \qquad
    \timevar \in [\ts, \tf],
\end{equation}
where
\begin{equation}\label{eq:MapDef}
    y(\timevar) \in \R^{\nvar}, \qquad J \in \R^{\nvar \times \nvar}, \qquad f, g : \R^{\nvar} \rightarrow \R^{\nvar}.
\end{equation}
Here $f$ and $g$ are $p$ times continuously differentiable functions (or in the very least $f,g$ are $C^p$ in a tubular neighborhood of the solution $y(t)$) which implies $y$ is $p+1$ times continuously differentiable.
The matrix $J$ satisfies the one-sided Lipschitz condition
\begin{equation} \label{eq:one_sided_Lipschitz}
    y^T J y \leq 0, \qquad \forall y \in \R^{\nvar}.    
\end{equation}
Condition \cref{eq:one_sided_Lipschitz} restricts the eigenvalues of $J$ to the left half of the complex plane, however places no bound on their magnitude (cf. \cite[Section 2]{roberts2026runge}).  

Equations of the form \cref{eq:ODE} often arise from \emph{method of lines} discretizations of partial differential equations (PDE) where $J$ approximates the highest order spatial derivative.  Thus, we are interested in the stiff setting where the step size is much larger than the fastest time scale, given by the reciprocal of the largest eigenvalue of $J$ in magnitude.
In this regime, traditional Runge--Kutta methods may converge at rates lower than their classical order.

\subsection{Runge--Kutta Methods: Classical Order Versus Uniform Accuracy}

A Runge--Kutta method approximates $y(t_n) \approx y_n$ with $s$ stages as 
\begin{subequations} \label{eq:RK}
    \begin{align}
        \label{eq:RK:stages}
        Y_{n,i} &= y_n + \dt_n \sum_{j=1}^{s} a_{i,j} f(Y_{n,j}), \qquad i = 1, \dots, s, \\
        \label{eq:RK:step}
        y_{n+1} &= y_n + \dt_n \sum_{i=1}^{s} b_i f(Y_{n,i}).
    \end{align}
    Here, $\dt_n$ is the variable step size with time grid $\timevar_{n+1} = \timevar_{n} + h_n$ for $n = 0, 1, 2, \ldots$. The  scheme is then defined by its coefficients
    \begin{align*}
        A = (a_{i,j})_{i,j=1}^s,
        \quad
        b = (b_i)_{i=1}^s,
        \quad c = (c_i)_{i=1}^s = A \one,
        \quad
        \one = [1, \dots, 1]^T \in \R^{s}.
    \end{align*}
    A DIRK method has a lower triangular $A$, and there are a number of important special cases including SDIRK, EDIRK, and ESDIRK structures.
    We refer readers to \cite[Section 2]{kennedy2016diagonally} for an overview of these acronyms.
    To enable standard local error estimation and step size control \cite[Section II.4]{hairer1993solving},
    an embedded Runge--Kutta method computes a secondary solution using weights $\widehat{b} = (\widehat{b}_i)_{i=1}^s$:
    \begin{equation} \label{eq:RK:embedding}
        \widehat{y}_{n+1} = y_n + \dt_n \sum_{i=1}^{s} \widehat{b}_i f(Y_{n,i}).
    \end{equation}
\end{subequations}

Traditional Runge--Kutta approximation theory relies on expanding the local error in powers of $\dt_0$ via Taylor series \cite[Section II.2]{hairer1993solving}:
\begin{equation}
    \begin{split}
        y(t_1) - y_1
        &= \dt_0 \Big(1 - b^T \one\Big) f 
        + \dt_0^2 \Big( \tfrac{1}{2} - b^T c \Big) f' f 
        + \dt_0^3 \Big( \tfrac{1}{6} - \tfrac{1}{2}b^T c^2 \Big) f''(f, f) \\ \label{eq:ClassicalExp}
        & \quad + \dt_0^3 \Big( \tfrac{1}{6} - b^T A c \Big) f' f' f + \ldots + h_0^{p+1} e_{p+1}(h_0).
    \end{split}
\end{equation}
Here, $e_{p+1}$ denotes the error term, and for brevity, we suppress the argument $y_0$ in $f$ and its derivatives. Setting the parenthetical terms in \cref{eq:ClassicalExp} to zero yields the classical order conditions. This provides a pathway to construct Runge--Kutta schemes with high classical order $p$, e.g., by enforcing $y(\timevar_1) - y_1$ to be $\order{\dt_0^{p+1}}$.

The problem with \cref{eq:ClassicalExp} is $e_{p+1}$ contains powers of $J$ which can be disproportionately large.
The terms with the smallest power of $\dt_0$ may not dominate the error but rather the terms in which $f$ and $f'$ appear the most in the elementary differential.
When \cref{eq:ODE} arises as a discretization of an initial boundary value problem with time dependent boundary conditions, the mechanism exists for a reduction of order in the local error \cite{rosales2024spatial}.

Recent work \cite{roberts2026runge} derived an alternative expansion for the local error. The key novelty is that the error $\widehat{e}_{p+1}$ remains uniformly bounded in the (stiff) variable $Z \coloneqq \dt_0 J$:
\begin{equation} \label{eq:StiffExp}
    \begin{split}
        y(t_1) - y_1
        &= \dt_0 \, \qo{1} \, y'
        + \dt_0^2 \, \qo{2} \, y'' 
        + \dt_0^3 \, \qo{3} \, y^{(3)} \\
        & \quad + \dt_0^2 \Big(b^T \! \otimes Z\Big) \Big(I - A \otimes Z\Big)^{\!-1} \! \Big( \so{2} \otimes y'' \Big) \\
        &\quad + \dt_0^3 \Big(b^T \! \otimes Z\Big) \Big(I - A \otimes Z\Big)^{\!-1} \! \Big( \so{3} \otimes y^{(3)} \Big) \\
        & \quad + \dt_0^3 \Big(b^T \! \otimes I\Big) \Big(I - A \otimes Z\Big)^{\!-1} \! \Big(I \otimes g'\Big)\Big(I - A \otimes Z\Big)^{\!-1} \! \Big( \so{2} \otimes y'' \Big) \hspace{-.2em} \\ 
        & \quad + \dots + \dt_0^{p+1} \widehat{e}_{p+1}(h_0).
    \end{split}
\end{equation}
Here, we suppress the $t_0$ argument in $y$ and its derivatives, as well as the $y(t_0)$ argument for $g'$.
The Kronecker product is denoted by $\otimes$, and
\begin{equation} \label{eq:gamma}
    \qo{\ell} \coloneqq \frac{1}{\ell!} - \frac{b^T c^{\ell - 1}}{(\ell - 1)!}, \qquad 
    \so{\ell} \coloneqq \frac{c^{\ell}}{\ell !} - \frac{A c^{\ell-1}}{(\ell - 1)!},  \qquad 
    (\ell \geq 1),
\end{equation}
where $c^{\ell} = (c_1^{\ell}, c_2^{\ell}, \ldots, c_s^{\ell})^T$.
We note that \cref{eq:gamma} is related to standard Runge--Kutta simplifying assumptions \cite[p. 208]{hairer1993solving} as follows:
\begin{align*}
    B(q_1): \quad \qo{1} = \qo{2} = \dots = \qo{q_1} = 0, \\
    C(q_2): \quad \so{1} = \so{2} = \dots = \so{q_2} = 0.
\end{align*}

The expansion in \cref{eq:StiffExp} provides a pathway to derive \emph{stiff semilinear order conditions}: pick the Runge--Kutta coefficients so that the local error in \cref{eq:StiffExp} vanishes to order $p+1$.
It is clear this can be done through the $p$-th \emph{stage order} conditions $B(p)$ and $C(p)$.
Stage order \cite[Section IV.15]{hairer1996solving}, however, is overly restrictive; it is limited to two for DIRK methods.
We will show below that the semilinear order conditions are less restrictive and admit solutions with a DIRK structure.

\subsection{Order Reduction and Work in Context}

The order reduction phenomenon has a long history of analysis and mitigation techniques.
A large portion of the literature has focused on linear ODEs \cite{prothero1974stability,ostermann1992runge,rang2014analysis,ketcheson2020dirk,roberts2022eliminating,biswas2023explicit,ArranzSimonPalencia2025, ArranzSimonCanoPalencia2026} or approaches that modify ODEs \cite{carpenter1993stability,pathria1997correct} (often the boundary condition in the case of PDEs) so that it is amenable to a particular Runge--Kutta method.
Works that consider semilinear \cite{burrage1986study,auzinger1992extension,calvo2000runge} and nonlinear problems \cite{frank1985order,dekker1984stability,burrage1987order,hairer1996solving} typically relate the order of convergence to the stage order of the Runge--Kutta method.

Some works have considered approaches to mitigate order reduction for classes of nonlinear problems without resorting to high stage order.
A sharp order condition theory for Runge--Kutta methods applied to stiff semilinear ODEs was developed in \cite{roberts2026runge}, and the current paper serves as a companion to this.
An alternative approach based on rational methods was proposed in \cite{ArranzSimonCanoPalencia2025semi} for this class of problems.
Skvortsov \cite{skvortsov2003accuracy,skvortsov2010model} constructed DIRK methods that exactly integrate nonlinear versions of the Prothero--Robinson problem, among other model problems.
In \cite{BoscarinoRusso2009,boscarino2007error}, the authors considered implicit-explicit Runge--Kutta methods satisfying order conditions for differential-algebraic equations which arise from singular perturbation analysis.
Although these conditions are necessary to eliminate order reduction and often work well in practice, they are not sufficient.
In the context of exponential integrators, high order schemes satisfying stiff order conditions have been proposed in \cite{LuanOstermann2014,luan2021efficient,luan2025sixth}.

The effect of order reduction on error estimation and step size adaptivity is scarcely explored \cite{donald1992implications,gustafsson1994control} \cite[pp.~113--114]{hairer1996solving}.
Similarly, there are few attempts to mitigate order reduction when deriving embedded Runge--Kutta methods.
Most rely on high stage order \cite{kvaerno2004singly,kennedy2016diagonally}; however, at least one work has considered less restrictive conditions for linear ODEs \cite{rang2014analysis}.

\subsection{Contributions and Outline of the Paper}

There are four primary contributions of this paper:

\begin{enumerate}
    \item 
    After introducing stiff semilinear order conditions in \cref{Sec:BackgroundSemilinearConditions}, we show that they can be satisfied through a set of algebraic matrix equations in \cref{Sec:AlgebraicRiccatiEq}.
    
    \item 
    In \cref{sec:derivations}, we devise new fourth and fifth order DIRK methods that satisfy the stiff semilinear order conditions and have optimized error constants and desirable stability properties such as L-stability.
    %

    \item Through extensive numerical experiments in \cref{sec:numerical_results}, we show that the new methods converge at their theoretically predicted orders when applied to semilinear ODEs and PDEs. Moreover, the methods alleviate order reduction for nonlinear problems not covered by the theory.
    
    \item 
    In \cref{sec:adaptive_time_stepping}, we demonstrate that classical embedded methods can suffer from order reduction, underestimate the local truncation error, and fail to meet user-specified tolerances.
    Embeddings in our new schemes recover the full order, allowing for more accurate and efficient adaptive time-stepping.
\end{enumerate}

We finish with our concluding remarks in \cref{sec:conclusion}.

\section{Background and Semilinear Order Conditions}
\label{Sec:BackgroundSemilinearConditions}

In this section, we outline the semilinear order conditions and convergence results adapted from \cite{roberts2026runge}. While the theory allows for a general inner product space on $\R^\nvar$, for simplicity we state the results with a normalized $\ell^2$ inner product
\begin{align*}
    \langle u, v \rangle = \frac{1}{\nvar} \sum_{j=1}^{\nvar} u_j v_j, \qquad u,v \in \R^\nvar.
\end{align*}
When $u$ is a discrete approximation to an $L^2$ Riemann integrable function, $\|u\|$ scales with the $L^2$ integral norm.

The convergence theorem in this section assumes the functions in \cref{eq:ODE} are bounded in the sense that there are constants $M, L > 0$ for which
\begin{equation} \label{eq:ODE_assumptions:nonlinear}
    \norm{g(y) - g(z)} \leq L \norm{y-z}
    \quad \forall \,  y, z \in \R^N,
\end{equation}
and
\begin{equation}\label{eq:ODE_assumptions:diff}
    \begin{alignedat}{3}
        \norm{y^{(k)}(\timevar)} &\leq M
        &\quad &\forall \, \timevar \in [\ts, \tf],
        &\quad k &= 1, \dots, \psl+1, \\
        \norm{g^{(k)}(y)} &\leq M
        &\quad &\forall \, y \in \R^N,
        &\quad k &= 1, \dots, \psl.
    \end{alignedat}
\end{equation}
As mentioned in the introduction, the domain $\R^N$ in \cref{eq:ODE_assumptions:nonlinear} and \cref{eq:ODE_assumptions:diff} can be replaced with tubular neighborhood of the solution at the expense of additional restrictions on $\dt_n$. 

The convergence theorem also makes use of the following assumptions on the Runge--Kutta scheme. The matrix $I - z A$ is assumed to be non-singular for all $z \in \C^{-} \coloneqq \{z \in \C : \Re(z) \leq 0 \}$ and 
\begin{alignat}{2}\label{eq:RK_Stability}
    (I - zA)^{-1}  & \qquad \textrm{and} \qquad &   z b^T (I - zA)^{-1} 
\end{alignat}
are uniformly bounded in $\C^{-}$.
Conditions in \cref{eq:RK_Stability} are also referred to as AS- and ASI-stability \cite[Def. 3.1--3.2]{burrage1986study}.

A Runge--Kutta scheme is said to have \emph{semilinear order} $\psl$ (cf. \cite[Definition 3]{roberts2026runge}) provided it satisfies all conditions in \cref{tab:order_conditions} up to label $\psl$ (e.g., $\psl = 3$ schemes satisfy conditions 1a, 2a, 3a, 3b).
The matrix $C$ used in the table is defined as
\begin{align}\label{eq:MatrixC}
     C \coloneqq \diag(c_1, \dots, c_s).
\end{align}

\begin{table}[ht!]
    \centering
    \begin{tabular}{r|l|l}
        Label & Order Condition $(\forall i_1, i_2, i_3 \in \{0, \dots, s - 1 \})$ & Implied By \\ \hline
        1a & $0 = 1 - b^T \one$ & Classical order 1 \\ \hline
        2a & $0 = \frac{1}{2} - b^T c = b^T A^{i_1} \left( \frac{c^{2}}{2} - A c \right)$ & Stage order 2 \\ \hline
        3a & $0 = \frac{1}{6} - \frac{b^T c^2}{2} =  b^T A^{i_1} \left( \frac{c^{3}}{6} - \frac{A c^2}{2} \right)$ & Stage order 3 \\
        3b & $0 = b^T A^{i_1 + i_2} \left( \frac{c^{2}}{2} - A c \right)$ & 2a \\ \hline
        4a & $0 = \frac{1}{24} - \frac{b^T c^3}{6} =  b^T A^{i_1} \left( \frac{c^{4}}{24} - \frac{A c^3}{6} \right)$ & Stage order 4 \\
        4b & $0 = b^T A^{i_1} C A^{i_2} \left( \frac{c^{2}}{2} - A c \right)$ & Stage order 2 \\
        4c & $0 = b^T A^{i_1 + i_2} \left( \frac{c^{3}}{6} - \frac{A c^2}{2} \right)$ & 3a \\
        4d & $0 = b^T A^{i_1 + i_2 + i_3 + 1} \left( \frac{c^{2}}{2} - A c \right)$ & 2a \\
        %
    \end{tabular}
    \caption{There are eight semilinear order conditions up to order four, however conditions 3b, 4c, and 4d are redundant because they are implied by lower order conditions. Note that each condition must hold over all powers of the matrix $A$. When $i_1 = i_2 = i_3 = 0$, these are classical order conditions in the form of Albrecht \cite{albrecht1996runge}.}
    \label{tab:order_conditions}
\end{table}


We then have the following result.

\begin{theorem}[Uniform-in-$J$ Error Estimates \cite{roberts2026runge}]\label{thm:MainLTE} 
    Let $(A,b,c)$ be a Runge--Kutta scheme with semilinear order $\psl$ (\cref{tab:order_conditions}) and the boundedness condition in \cref{eq:RK_Stability}.  Assume that \cref{eq:one_sided_Lipschitz}, \cref{eq:ODE_assumptions:nonlinear} and \cref{eq:ODE_assumptions:diff} hold.
    
    Then there are constants $D$ and $\widetilde{\dt}$ depending only on the $L$ and $M$, and the method coefficients (but not on $J, Z$) for which the local error \cref{eq:StiffExp} satisfies
    \begin{equation}\label{eq:LTE_ErrorEst}
        \norm{y(t_1) - y_1} \leq D \dt_0^{\psl + 1}, \qquad \forall \dt_0 \in [0, \widetilde{\dt}) .
    \end{equation}    

    Furthermore, when \cref{eq:RK} is applied to \cref{eq:ODE} with equal step sizes $\dt_n = \dt$, the global error satisfies 
    \begin{equation} \label{eq:global_error}
        \begin{split}
            \norm{y(t_n) - y_n} &\leq \overline{D} \dt^{\psl},
            \qquad            
            \forall \dt \in [0, \overline{\dt}), \quad
            n \dt \leq \tf-\ts. 
        \end{split}        
    \end{equation}
    Again, $\overline{D}, \overline{\dt}$ are independent of $J$ or $Z$. 
\end{theorem}

\begin{remark}[Superconvergence {\cite[Theorem 5]{roberts2026runge}}]
\label{rmk:supercvg} 
The order in \cref{eq:global_error} is improved to $\psl+1$ for problems where $y$ and $g$ are $\psl+2$ and $\psl+1$ times differentiable, respectively, and the Runge--Kutta scheme also satisfies: (i) the $\psl+1$ classical order conditions, and (ii) the technical condition 
\begin{align}\label{eq:TechnicalCond}
    \lim_{z\rightarrow \infty} R(z) \neq 1, \qquad \textrm{and} \qquad z^{-1}(1 - R(z)) \quad \textrm{has no zeros in } \C^{-}.
\end{align}
Here $R(z) \coloneqq 1 + b^T (I - z A)^{-1} \one$ is the linear stability function.
\end{remark}



\section{The Semilinear Order Conditions in Matrix Form}
\label{Sec:AlgebraicRiccatiEq}

In this section, we convert the conditions in \cref{tab:order_conditions} into a set of matrix equations to enable the construction of schemes with semilinear order up to four. The idea is to interpret the conditions in \cref{tab:order_conditions} as a set of orthogonality relations.  

To start, introduce the matrix 
\begin{align*}
    \Lambda = \begin{bmatrix}
        \Lambda_1 & \Lambda_2
    \end{bmatrix} \in \R^{s\times d} \qquad \textrm{with} \qquad \Lambda_1 \in \R^{s \times d_1}, \quad \Lambda_2 \in \R^{s \times d_2},
\end{align*}
where $d = d_1 + d_2$. At this point, we do not specify the dimension $d_1$ or $d_2$. We choose $\Lambda$ so that:
\begin{enumerate}
    \item[(P1)] The column spaces of $\Lambda_1$ and $\Lambda$ are $A$-invariant 
\begin{align}
    A \col(\Lambda) &\subseteq \col(\Lambda) , \\
     A \col(\Lambda_1) &\subseteq \col(\Lambda_1) .
\end{align}
    \item[(P2)] The column spaces of $\Lambda_1$ and $\Lambda$ contain 
    \begin{align}
        \so{2} \in \col(\Lambda_1) \qquad \textrm{and} 
        \qquad 
        \so{3}, \so{4} \in \col(\Lambda).
    \end{align}
    \item[(P3)] The matrix $C$ from \cref{eq:MatrixC} maps $\col(\Lambda_1)$ into $\col(\Lambda)$
    \begin{align*}
        C \,\col(\Lambda_1) \subseteq \col(\Lambda) .
    \end{align*}
    \item[(P4)] The vector $b$ is orthogonal to $\col(\Lambda)$, i.e., $b \perp \col(\Lambda)$.
\end{enumerate}
As we will show below, together with the simplifying assumption $B(4)$, conditions (P1)--(P4) are sufficient to satisfy the semilinear conditions of order 4 in  \cref{tab:order_conditions}.  

Conditions (P1)--(P4) are also equivalent to the existence of matrices
\begin{align*}
    X \coloneqq \begin{bmatrix}
        X_{11} & X_{12}  \\
        0 & X_{22} 
    \end{bmatrix} \in \R^{d \times d}, \qquad
    W \coloneqq \begin{bmatrix}
        W_{11} & W_{12} \\
        0 & W_{22} 
    \end{bmatrix} \in \R^{d \times 3}, \qquad Y \in \R^{d \times d_1},
\end{align*}
where
\begin{alignat}{3}\nonumber 
    X_{11} &\in \R^{d_1\times d_1}& \qquad   
    X_{12} &\in \R^{d_1\times d_2}& \qquad 
    X_{22} &\in \R^{d_2 \times d_2}, \\ \nonumber 
    W_{11} &\in \R^{d_1} &
    W_{12} &\in \R^{d_1 \times 2} &
    W_{22} &\in \R^{d_2 \times 2},
\end{alignat}
satisfying the following system of equations:
\begin{subequations}\label{eq:MatSystem}
\begin{empheq}[left=\empheqlbrace]{align}
    \label{eq:MatSystem:Eq1}
    A \Lambda &= \Lambda X \\
    \label{eq:MatSystem:Eq2}
    T_{4} &= \Lambda W \\
    \label{eq:MatSystem:Eq3}
    C \Lambda_1 &= \Lambda Y \\
    \label{eq:MatSystem:Eq4}
    b^T \Lambda &= 0.
\end{empheq}   
\end{subequations}
In \cref{eq:MatSystem:Eq2}, the matrix $T_j$ is defined as
\begin{align*}
    T_j \coloneqq \begin{bmatrix}
        \so{2}, & \so{3}, & \cdots, & \so{j}
    \end{bmatrix}.
\end{align*}

The following theorem formalizes the sufficiency of $B(4)$ and \cref{eq:MatSystem} for satisfying the semilinear order 4 conditions.
\begin{theorem}[Matrix equations for semilinear order four]
    Let $(A,b,c)$ be a Runge--Kutta scheme satisfying the condition $B(4)$. If for some choice of $d_1, d_2$, there exists a solution $(\Lambda, X, W, Y)$ to the system of equations \cref{eq:MatSystem}, then the scheme has semilinear order $\psl = 4$. That is, the scheme satisfies the order conditions in \cref{tab:order_conditions}.
\end{theorem}

\begin{proof} Condition $B(4)$ is a trivial restatement of the leftmost order condition in \cref{tab:order_conditions} lines 1a, 2a, 3a, 4a. \Cref{tab:order_conditions} line 4b follows since 
\begin{alignat*}{2}
    b^T A^{i_1} C A^{i_2} \so{2} &= b^T A^{i_1} C \Lambda_1 u & \qquad &\textrm{($A^{i_2}\so{2} = \Lambda_1 u$ for some $u$ by \cref{eq:MatSystem:Eq1,eq:MatSystem:Eq2})} \\
    &= b^T A^{i_1} \Lambda Y u  & \qquad &\textrm{(by \cref{eq:MatSystem:Eq3})} \\
    &= b^T \Lambda X^{i_1}Yu & \qquad &\textrm{(by \cref{eq:MatSystem:Eq1})} \\
    &= 0. & \qquad &\textrm{(by \cref{eq:MatSystem:Eq4})} 
\end{alignat*}
The remaining conditions (i.e., \cref{tab:order_conditions} lines 2a, 3a, 3b, 4a, 4c, 4d) follow since, for any power $j$
\begin{equation*}
    b^T A^j T_4
    = b^T A^j \Lambda W
    = (b^T \Lambda) X^j W
    = 0.
\end{equation*}
\end{proof}

\begin{corollary}[Semilinear conditions of order 3] A Runge--Kutta scheme has semilinear order three if $B(3)$ holds and there exist $\Lambda \in \R^{s\times d}$, $X \in \R^{d \times d}$, and $W \in \R^{d \times 2}$ for which the following (subset) of equations hold
\begin{equation} \label{eq:3rd_order_matrix}
    A \Lambda = \Lambda X, \quad 
    T_3 = \Lambda W, \quad
    b^T \Lambda = 0.
\end{equation}    
Note that no block structure on $\Lambda, X$ or $W$ is needed for the $\psl = 3$ equations.
\end{corollary}

\section{Construction of methods with high semilinear order}
\label{sec:derivations}
In this section, we construct Runge--Kutta schemes that satisfy the semilinear order conditions up to order four, but with superconvergence (see \cref{rmk:supercvg}), have up to fifth order global error.
The methods also have desirable stability and accuracy properties, making them suitable for solving stiff, semilinear ODEs.  

We adopt the method naming convention \textit{TYPE}-($s$, $p$, $\psl$), where \textit{TYPE} denotes the method structure, e.g., ESDIRK or EDIRK.
Furthermore, $s$ is the number of stages, $p$ is the classical order, and $\psl$ is the semilinear order. 

A variety of schemes already exist in the literature with low semilinear order larger than one because the conditions in \cref{tab:order_conditions} can be satisfied via stage order conditions and reduce to the weak stage order conditions when $\psl \leq 3$ (see \cite[Remark 3]{roberts2026runge}).

\begin{itemize}\setlength\itemsep{0em}
    \item A range of EDIRK methods with stage order two can be found in \cite{kennedy2016diagonally}.
    With superconvergence, these can attain up to third order global error when solving \cref{eq:ODE}.

    \item A complete parameterization of DIRK-$(2,2,3)$, and DIRK-$(3,3,3)$ methods can be found in \cite[Section~7]{biswas2022algebraic}.

    \item A complete parameterization of ERK-$(3,2,2)$, ERK-$(4,3,2)$ and ERK-$(5,3,3)$ methods can be found in \cite[Section~3]{biswas2023explicit}.

    \item DIRK-$(4,3,3)$ and DIRK-$(6,4,3)$ methods that are L-stable and stiffly accurate can be found in \cite[p.~459]{ketcheson2020dirk}.
\end{itemize}

Given the extensive catalog of existing optimized methods with $\psl = 2$, we do not propose new schemes of this type.
\Cref{subsec:symbolic_derivation} details the derivation of ESDIRK methods with $\psl \in \{3, 4\}$ and stage order two, while \cref{subsec:numeric_derivation} details the derivation of EDIRK methods with $\psl = 4$ and stage order one.

A summary of the properties of our new methods can be found in \cref{tab:method_properties}.
As a point of comparison, both for analysis and numerical experiments, we consider two methods from the literature with $\psl=1$.
Following our naming conventions, these are referred to as SDIRK-(5,4,1) and SDIRK-(6,5,1).
The former, which is named SDIRK4M in \cite{kennedy2016diagonally}, is selected for its optimized error, L-stability, and quality embedding.
The latter, which appears in \cite{ismail1998embedded}, is one of the only embedded pairs we found in the literature with $p = 5$ and $\psl = 1$.

\subsection{Construction of ESDIRK-(8,4,3) and ESDIRK-(10,5,4)}
\label{subsec:symbolic_derivation}
Through symbolic computation, we construct here two stiffly accurate, L-stable schemes, with embeddings. The schemes are exact solutions of the classical order conditions and semilinear order conditions and complement existing schemes \cite{biswas2023design}. For simplicity, \ref{app:ESDIRK-(843)} and \ref{app:ESDIRK-(1054)} report the coefficients of the new schemes, ESDIRK-(8,4,3) and ESDIRK-(10,5,4), as rational approximations accurate to quadruple-precision. 

\medskip
\noindent \emph{Construction of} ESDIRK-(8,4,3).  One of the first design questions to consider is: How many stages should a method with $p=4$ and $\psl = 3$ have?
We have a lower bound of $s \geq 5$ from \cite[Theorem 3.2]{biswas2022algebraic} (take $(p, q, \kappa, \sigma) = (4, 3, 1, 1)$, where $\kappa = 1$ is allowed as the theorem holds if $A$ is singular provided $R(\infty) = 0$).
This lower bound is likely not sharp as we impose the additional ESDIRK structural constraint of $a_{2,2} = \ldots = a_{s,s} \eqqcolon d$.
Indeed, through symbolic computation, we did not find L-stable ESDIRK-($s$,4,3) schemes with $s \in \{5, 6\}$, but they exist for $s = 7$.
The extra degrees of freedom from $s=8$, however, allow for a smaller principal error and better internal stability \cite[Section~2.6]{kennedy2016diagonally} in the stiff limit.

To finalize our construction of an ESDIRK-(8,4,3), we use Mathematica and the Integreat library \cite{roberts2025integreat} to symbolically solve the third order conditions \cref{eq:3rd_order_matrix} and the following conditions:
\begin{alignat*}{2}
    B(5), \quad C(2), \quad b^T C A c^2 = \frac{1}{15}, & \qquad & \text{(Classical order conditions)}, \\
    %
    %
    \lim_{z \to \infty} e_i^T (I - z A)^{-1} \one = 0, \quad i = 4, \dots, 8, & \qquad & \text{(Internal stability)}, \\
    e_7^T A c^{i-1} = \frac{1}{i}, \quad i = 1, 2, 3, & \qquad & \text{(Embedding)}.
\end{alignat*}
Here, $e_i \in \R^{s}$ is the $i$-th unit vector.  The L-stability condition restricts the diagonal of $A$ to $d \in [0.24799, 0.67604]$ (see \cite[Table~5]{kennedy2016diagonally}). As in \cite[Section~8.2]{kennedy2019diagonally}, we set $d = \frac{31}{125}$ as a rational approximation to the interval's lower bound.

We use the penultimate stage, $Y_{n,7}$, as an internal embedded solution with classical order three and semilinear order two (both one lower than the primary method).
Therefore, $\widehat{b}_i = a_{7,i}$ for $i = 1, \dots, 8$, and $Y_{n,8} - Y_{n,7}$ can be used as an estimate of the local truncation error.
The three remaining free parameters are used to optimize the quality of the embedding as measured by $B^{(p+1)}$, $C^{(p+1)}$, and $E^{(p+1)}$ defined in \cite[Section~2.3]{kennedy2016diagonally}.

\medskip
\noindent \emph{Construction of} ESDIRK-(10,5,4).
We did not find an ESDIRK method with $p = 5$ and $\psl = 4$ in fewer than ten stages.
Using this minimum value of $s$, we symbolically solve the semilinear order conditions \cref{eq:MatSystem} and
\begin{alignat*}{2}
    B(6), \quad C(2), \quad b^T C A c^2 = \frac{1}{15}, & \qquad & \text{(Classical order conditions)}, \\
    \lim_{z \to \infty} e_i^T (I - z A)^{-1} \one = 0, \quad i = 3, \dots, 10, & \qquad & \text{(Internal stability)}.
\end{alignat*}
For L-stability, $d$ must be the root of
\begin{equation*}
    0 = 120 d^5-600 d^4+600 d^3-200 d^2+25 d -1
\end{equation*}
closest to $0.278$ as in \cite[Table~5]{kennedy2016diagonally}.

For an embedding, it is infeasible to use the penultimate stage as we did with ESDIRK-(8,4,3), and instead we use the general form \cref{eq:RK:embedding}.
All but one embedded coefficient is determined by the conditions for classical order four, semilinear order three, and A-stability.
The remaining 11 free parameters and several solution branches are used to optimize the principal error and embedding quality.

\subsection{Construction of EDIRK-(7,4,4) and EDIRK-(19,5,4)}
\label{subsec:numeric_derivation}
The $\psl = 4$ scheme in the previous subsection uses the classical $C(2)$ simplifying condition (note that $C(2)$ implies line 4b in \cref{tab:order_conditions} and \cref{eq:MatSystem:Eq3}). Due to the restrictive nature of stage order, it is of theoretical interest to construct schemes without $C(2)$, which we do here.

\medskip
\noindent \emph{Construction of} EDIRK-(7,4,4). Here we use numerical optimization in MATLAB: the order conditions and semilinear order conditions \cref{eq:MatSystem} are set as equality constraints and positivity of $c_i$ and $a_{i,i}$ for $i = 2, \dots, s$ as inequality constraints. 

An A-stability constraint is imposed by requiring the absolute value of the stability function to be less than or equal to one at a finite set of points on the imaginary axis, following the approach described in \cite{biswas2023design}.
This is a relaxation of the full A-stability condition; however, we verify post-hoc that any method found is indeed A-stable (see also \cite{juhlshirokoff2024}).
We utilize the \texttt{fmincon} algorithm to search for feasible methods, performing millions of evaluations with a constant objective function, starting from random initial guesses.
Ultimately, we select the method with the smallest principal error norm from those that satisfy all constraints, denoted here as the EDIRK-(7,4,4) method.
The coefficients refined to quadruple-precision with Mathematica are provided in \ref{app:EDIRK-(744)}.

With this derivation strategy and the very high dimensional search space, EDIRK-(7,4,4) is not optimal in terms of principal error and other properties.
It is primarily of theoretical interest and serves as a foundation for higher order schemes.
For this reason, we do not equip it with an embedded method.

We observe that the coefficients of EDIRK-(7,4,4) satisfy \cref{eq:MatSystem} for matrices of the form
\begin{align*}
    \Lambda = [\underbrace{\so{2}}_{\Lambda_1}, \underbrace{\so{3}}_{\Lambda_2}],
    \quad
    X = \begin{bmatrix}
        a_{4,4} & \star \\
        0 & a_{2,2}
    \end{bmatrix},
    \quad
    Y = \begin{bmatrix}
        c_4 & 0
    \end{bmatrix},
    \quad
    W = \begin{bmatrix}
        1 & 0 & \star \\
        0 & 1 & 3 a_{2,2}
    \end{bmatrix}.
\end{align*}
Note that $\so{2}$ is an eigenvector of both $C$ and $A$.

\medskip
\noindent \emph{Construction of} EDIRK-(19,5,4). A natural question is: Are there methods with $p = \psl + 1 = 5$ that do not have stage order two? A simple way to answer this affirmatively is to apply Richardson's extrapolation \cite[Section II.9]{hairer1993solving} to the EDIRK-(7,4,4) scheme.
This increases the classical order by one while maintaining $\psl = 4$.
EDIRK-(7,4,4) and any extrapolation of it have stage order exactly one.
Using the simplest step-number sequence of $\{1, 2\}$ for the extrapolation results in a 19 stage method.
At such a high number of stages, we cannot recommend this EDIRK-(19,5,4) method for practical computations, but it answers the theoretical question above positively.

\section{Numerical results: Convergence}
\label{sec:numerical_results}
We devise a series of test cases that reveal in which way the newly developed numerical methods alleviate order reduction, and how they perform in comparison to standard methods. We start with a semilinear variant of the classical Prothero--Robinson ODE problem (\cref{subsec:numerics_SLPR}).
Then we consider several PDE test problems that successively deviate further from the assumptions of \cref{Sec:BackgroundSemilinearConditions}, ranging from semilinear PDEs (\cref{subsec:numerics_SL_advection,subsec:Allen-Cahn}), to a nonlinear PDE whose highest derivative operators are linear (\cref{subsec:viscous_Burgers}). We then close with an ODE test case whose stiff terms are fully nonlinear (\cref{subsec:van_der_Pol}).
Throughout this and the next section, for notational economy, all methods from \cref{tab:method_properties} are denoted simply by their triplet $(s,p,\psl)$.
The code to reproduce all the numerical results 
is available in \cite{biswas2026semilinearRepro}.

In this section, all nonlinear systems arising in the Runge--Kutta schemes' stage solutions are solved via Newton iteration to high accuracy, so that the nonlinear solver error does not affect the temporal convergence curves.

\subsection{A semilinear variant of the Prothero--Robinson problem}
\label{subsec:numerics_SLPR}
Building on the idea of Prothero and Robinson \cite{prothero1974stability}, we construct a one-parameter family of test problems whose smooth closed-form solution is independent of a stiffness parameter $\lambda$; however, unlike \cite{prothero1974stability} the problem here is semilinear.
Suppose we have a non-stiff, nonlinear ODE $u'(\timevar) = \phi(u(\timevar))$ with known solution $u(\timevar)$.  The semilinear Prothero--Robinson problem is defined as
\begin{equation} \label{eq:SLPR}
    y'(\timevar) = \lambda (y(\timevar) - u(\timevar)) + \phi(y(\timevar)), \quad
    y(\ts) = u(\ts),
\end{equation}
with exact solution $y(\timevar) = u(\timevar)$.
Note that a similar construction has been presented by Skvortsov to generate nonlinear model problems \cite{skvortsov2003accuracy}.

For our numerical results, we choose the exact solution $u(\timevar) = \sqrt{1 + \timevar^2} - t$ and right-hand side $\phi(u(\timevar)) = \frac{-2 u(\timevar)^2}{1 + u(\timevar)^2}$, yielding the semilinear test problem
\begin{equation}\label{eq:SemiLinPR}
    y'(\timevar) = \lambda \left( y(\timevar) - \sqrt{1 + \timevar^2} + \timevar \right) - \frac{2 y(\timevar)^2}{1 + y(\timevar)^2}.
\end{equation}

\begin{figure}[htb]
	\begin{minipage}[b]{.32\textwidth}
		\includegraphics[width=\textwidth]{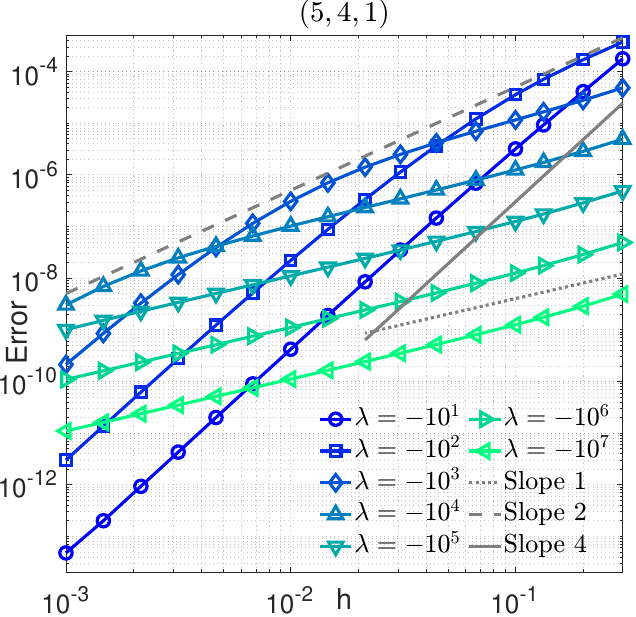}
	\end{minipage}
    \hfill
	\begin{minipage}[b]{.32\textwidth}
		\includegraphics[width=\textwidth]{Exp_figures/SL_PR_TC1_tf1.2_Convg_s5p4q1.pdf}
	\end{minipage}
    \hfill
	\begin{minipage}[b]{.32\textwidth}
		\includegraphics[width=\textwidth]{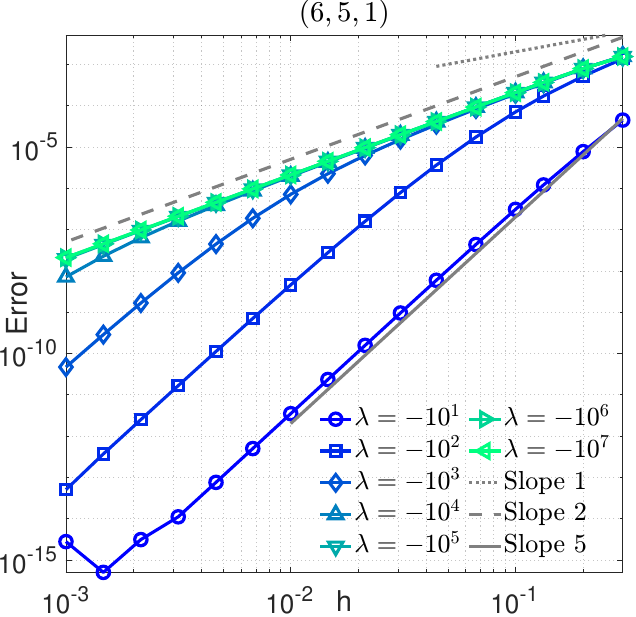}
	\end{minipage}

	\vspace{.5em} 
	
	\begin{minipage}[b]{.32\textwidth}
		\includegraphics[width=\textwidth]{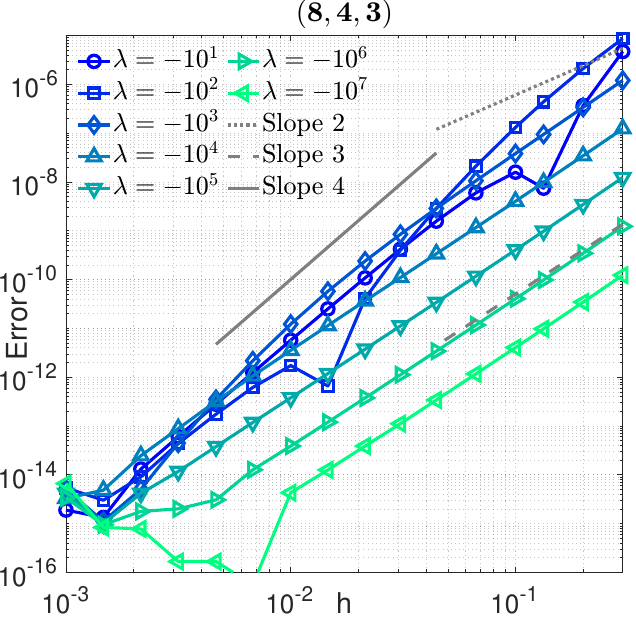}
	\end{minipage}
    \hfill
	\begin{minipage}[b]{.32\textwidth}
		\includegraphics[width=\textwidth]{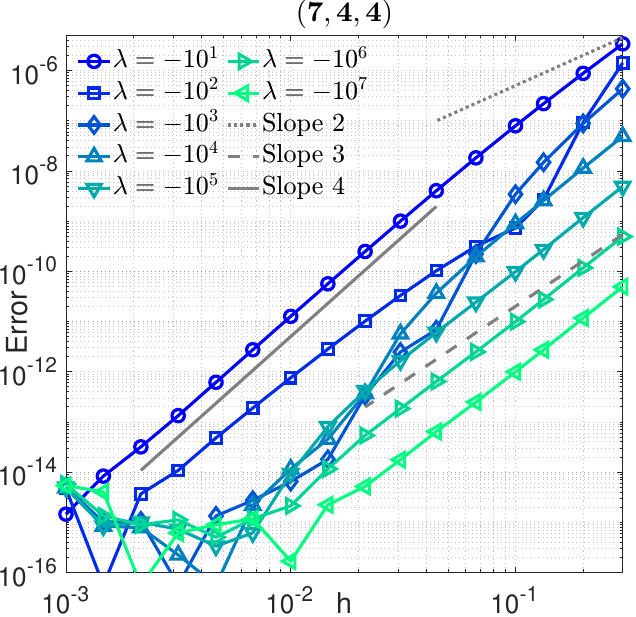}
	\end{minipage}
    \hfill
	\begin{minipage}[b]{.32\textwidth}
		\includegraphics[width=\textwidth]{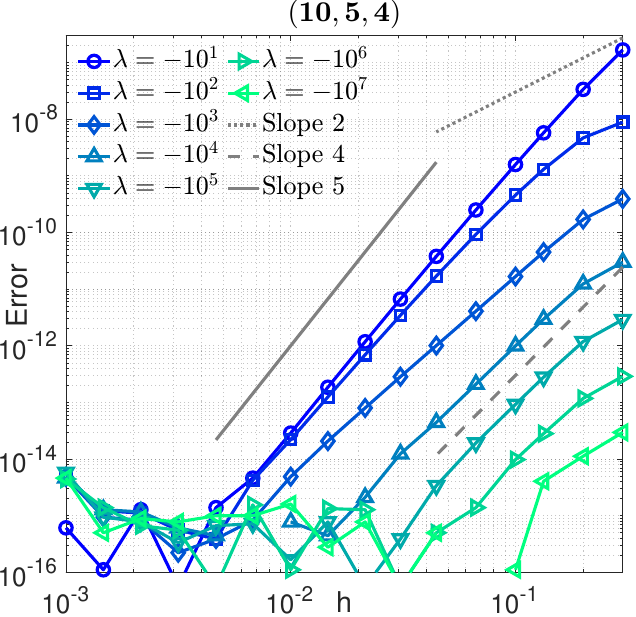}
	\end{minipage}
    \caption{Error convergence of different methods for the semilinear Prothero--Robinson problem \cref{eq:SemiLinPR}, shown for several values of the stiffness parameter $\lambda$ (corresponding to the different curves). The top row shows classical reference methods, while the bottom row displays new methods with high semilinear orders. Of these, the left and the right panels show methods with $\psl = p-1$, while the middle panel shows a fourth order method with $\psl = p$.}
 \label{fig:ErrConvg_SemiLinPR}
\end{figure}

Comparing different numerical methods, \cref{fig:ErrConvg_SemiLinPR} shows the resulting error convergence curves (in classical double-logarithmic scale) for the problem \cref{eq:SemiLinPR} at final time $t = 1.2$, using the stiffness parameter $\lambda = -10^\nu$ for $\nu \in \{1,\dots,7\}$.
%
The top row displays reference methods that satisfy the classical order conditions up to the respective orders ($p = 4$ for left and middle, and $p = 5$ on the right), but not any of the (other) semilinear order conditions.
These top row results show that the reference methods (i)~converge at their formal order $p$ in the non-stiff resolved regime ($\abs{\lambda\dt} \ll 1$) and they (ii)~clearly exhibit order reduction in the stiff regime ($\abs{\lambda\dt} \gg 1$), reducing to first order in $h$ for each $\lambda$. Furthermore, (iii)~the family of error curves, parameterized by $\lambda$, lies uniformly below a line of slope 2, matching the global (B-convergence) error bound established in \cite{roberts2026runge}. Another notable fact is that (iv)~the (5,4,1) method (being stiffly accurate) converges as $\lambda\to -\infty$ with $\abs{\lambda\dt} \gg 1$, while the non-stiffly accurate (6,5,1) method fails to establish such a convergence.

In the bottom row of \cref{fig:ErrConvg_SemiLinPR}, three new methods, which have high semilinear order, do not exhibit the characteristic order reduction L-shape that the classical methods show around $\lambda\dt \approx 1$. In particular, for sufficiently small step sizes, the new methods are notably more accurate than the reference schemes. At the same time, the new methods do not necessarily yield a uniform convergence rate across all step sizes. Instead, there is a line of slope $\min(p,\psl+1)$ below which all error curves lie for step sizes sufficiently small, independent of the stiffness parameter $\lambda$. In this test, all error curves stagnate around machine precision, and all new methods yield accurate results. 

\subsection{Methodology for PDE test problems}
\label{subsec:PDE_methodology}
In the next three subsections, we consider several PDE test problems. Consistent with much of the mathematical literature, we denote a PDE's solution by $u$, time by $t$, and space by $x$ (respectively $x_1$, $x_2$ in two space dimensions). The step size remains denoted by $\dt$.

A key aspect in evaluating the effectiveness of time integration methods for PDE problems is that a highly accurate spatial discretization is critical to effectively isolate the temporal error. Across the different problems, either spectral methods or high order spatial approximations with sufficiently small mesh sizes are used.


As the theory in \cref{Sec:BackgroundSemilinearConditions} allows for any inner-product-norm, we measure all spatial errors in a scaled $\ell^2$-norm which approximates the $L^2$-norm in space. Namely, if $y \in \R^{\nvar}$ with values of $y_j$ on a uniform grid with grid spacing $\Delta x = O(1/\nvar)$, then $\norm{y}_2^2 = \Delta x\sum_{j=1}^\nvar y_j^2$ (for the Chebyshev grids in \cref{subsec:Allen-Cahn}, appropriate weights get added in the summation). Norms involving the PDE solution use values of $u$ evaluated at the discrete grid points.

In all PDE test problems, we use spatial approximations that ensure that spatial errors are negligible (i.e., less than $3 \times 10^{-13}$). We estimate the spatial error by computing a reference solution to the method-of-lines (MOL) ODE via MATLAB's \texttt{ode45} with absolute tolerance $10^{-14}$ and relative tolerance $10^{-12}$. All results below report the numerical error of the fully discrete solution $U(\tf)$ against the PDE solution evaluated on the grid, $u(\cdot, \tf)$. This error is controlled by
\begin{equation}\label{eq:Triangle}
\|U(\tf)-u(\cdot,\tf)\|_{2} \le \underbrace{\|U(\tf)-u^{\mathrm{MOL}}(\tf)\|_{2}}_{\leq \overline{D} \dt^{\psl} \textrm{~from \cref{thm:MainLTE}} } + \underbrace{\|u^{\mathrm{MOL}}(\tf)-u(\cdot,\tf)\|_{2}}_{\le 3\times 10^{-13}},
\end{equation}
where $u^{\mathrm{MOL}}(t) \in \R^\nvar$ is the (exact) MOL solution. In the subsequent numerical tests of temporal errors, error curves then tend to stagnate at that semi-discretization error value of ${\sim}10^{-13}$ or below. In turn, for errors above this spatial error bound, the theory in \cref{Sec:BackgroundSemilinearConditions} controls the fully discrete error with respect to the MOL solution.

\subsection{Semilinear advection-reaction test problem}
\label{subsec:numerics_SL_advection}
As a first PDE test, we consider a PDE with linear advection and nonlinear reaction
\begin{equation}
\label{eq:advection-reaction_PDE}
    \begin{aligned}
        u_t + u_x & = 10(u-1)(2-u), \quad 0\leq x \leq 1,\quad 0 \leq t \leq t_\mathrm{f} , \\
        u(x,0) & = u_{0}(x), \\
        u(0,t) & = g_{0}(t).
    \end{aligned}
\end{equation}
This problem models a logistic growth (shifted to occur for $1<u<2$) of a profile moving to the right with speed 1. It is a variant of the benchmark equation $u_t +u_x = u^2$, in which Hundsdorfer and Verwer \cite{hundsdorfer2003numerical} demonstrated order reduction. 
Unlike test problems with manufactured solutions, problem \cref{eq:advection-reaction_PDE} is devoid of a forcing $f(x,t)$, so its solution exhibits its natural behavior, and order reduction effects are expected to predominantly emanate from inflow boundaries \cite{rosales2024spatial}. The problem fits the structure \cref{eq:ODE} because the linear $u_x$ is the stiffest term, while the nonlinear reaction term is less stiff. 

Solutions of \cref{eq:advection-reaction_PDE} can be constructed analytically. We choose the function
\begin{equation}
\label{eq:advection-reaction_solution}
U(x,t) = \frac{198 + 4\exp(10t) + (2\exp(10t)-1)\sin(4\pi(x-t))}
               {198 + 2\exp(10t) + (\exp(10t)-1)\sin(4\pi(x-t))},
\end{equation}
and set the initial condition $u_0(x) = U(x,0)$ and inflow boundary conditions $g_0(t) = U(0,t)$, yielding the true solution $u(x,t) = U(x,t)$. This solution transitions from $u\approx 1$ to $u\approx 2$ on a time scale $0\le t\le 1$, with sizable temporal rates of change and spatial variability occurring around $t\approx 0.5$. Therefore, for the convergence tests below we use $\tf = 0.5$.

On a uniform grid with spacing $\Delta x = 10^{-4}$, the spatial derivative is discretized using the fifth-order finite-difference scheme denoted by $(4,4,6\text{-}6\text{-}4,4,4)$ in \cite{carpenter1993stability}. This yields a spatial semi-discretization error of $1.4 \times 10^{-13}$, estimated as described in \cref{subsec:PDE_methodology}.


\begin{figure}[htb]
	\begin{minipage}[b]{.32\textwidth}
		\includegraphics[width=\textwidth]{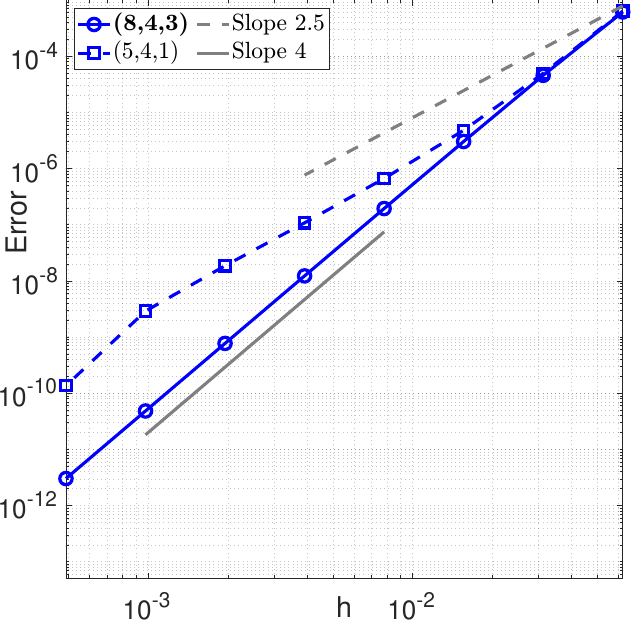}
	\end{minipage}
    \hfill
	\begin{minipage}[b]{.32\textwidth}
	\includegraphics[width=\textwidth]{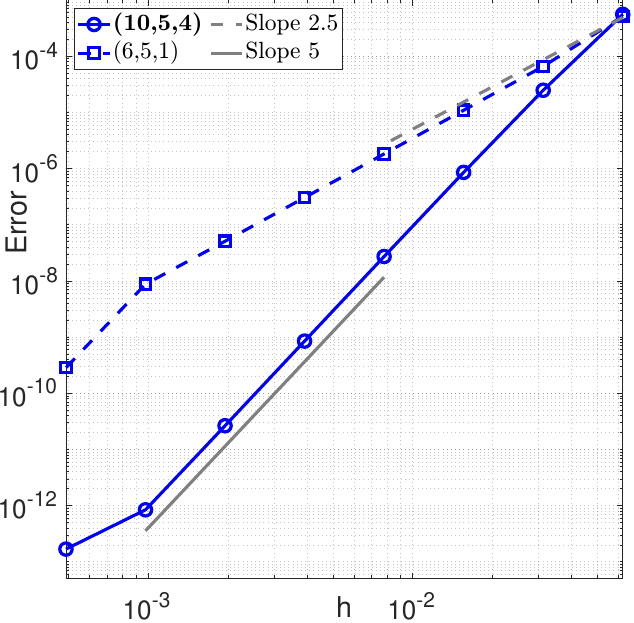}
	\end{minipage}
    \hfill
    \begin{minipage}[b]{.32\textwidth}
	\includegraphics[width=\textwidth]{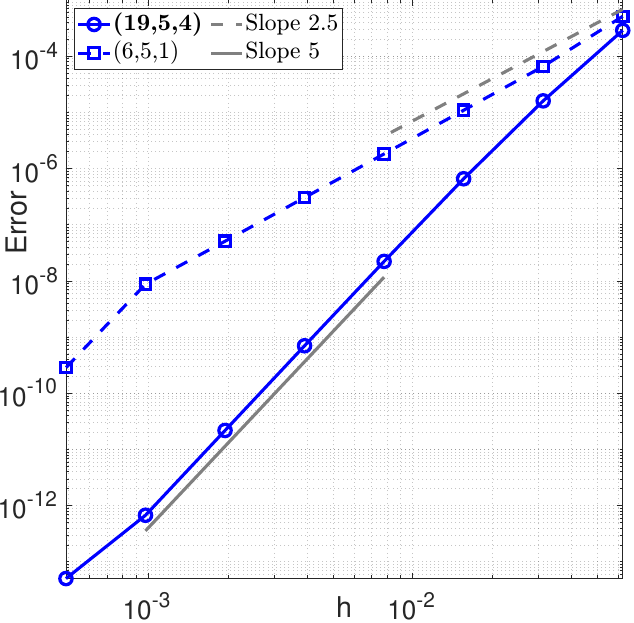}
	\end{minipage}
	\caption{Convergence in the scaled $\ell^2$-norm for the semilinear advection-reaction equation \cref{eq:advection-reaction_PDE}. In each panel, a $p$-th order method satisfying the semilinear order conditions up to $\psl = p-1$ is compared against a classical reference method of the same order $p$ but with $\psl = 1$. The new methods consistently exhibit the convergence orders established by the semilinear theory.}
	\label{fig:ErrConvg_SemiLinAdvReac}
\end{figure}

\Cref{fig:ErrConvg_SemiLinAdvReac} shows the convergence curves of the temporal errors, comparing the new methods against classical reference methods. In each panel, we can see that the reference method ((5,4,1) respectively (6,5,1)) does exhibit order reduction, reducing the error curve to a slope of 2.5 in the stiff regime (here: $h$ larger than $\approx 10^{-3}$). This error rate of 2.5 is to be expected, as follows. By asymptotic theory \cite{rosales2024spatial} one expects a numerical boundary layer of thickness $O(h)$ to arise in which the error is $O(h^2)$. For high order methods, this layer dominates the $L^2$ error, yielding
\begin{equation*}
\left(\int_0^1 ((U-u)(x,0.5))^2 \ud{x}\right)^{\!\frac{1}{2}}
= O\left((h \times (h^2)^2)^\frac{1}{2}\right) = O(h^{2.5}).
\end{equation*}
In contrast, the new methods with high semilinear order $\psl \ge p-1$ converge at the full order $p$ (until they plateau at the spatial approximation error). Convergence tests were conducted for all available methods, including those with $\psl = p$. The outcome is that $\psl = p-1$ generally suffices to achieve the full convergence order $p$; hence only such methods are displayed here.

\subsection{Allen--Cahn equation in two space dimensions}
\label{subsec:Allen-Cahn}
We now study the numerical methods on an initial--boundary-value problem in two spatial dimensions. We consider the Allen--Cahn equation
\begin{equation}\label{eq:Allen_Cahn_2D}
    u_t = \alpha \nabla^2u + \beta(u-u^3) + \psi(x_1,x_2,t), \quad  \ 0\leq x_1,x_2\leq 1, \quad 0 \leq t \leq 0.5,
\end{equation}
with parameter choices $\alpha = 0.1$ and $\beta = 3$. The initial conditions, the time-dependent Dirichlet boundary conditions, and the source term $\psi(x_1,x_2,t)$ are such that the manufactured solution is $u(x_1,x_2,t) = 2 + \sin(2\pi(x_1-t)) \cos(3\pi(x_2-t))$. Equation \cref{eq:Allen_Cahn_2D} is a nonlinear reaction-diffusion PDE, which is used, for example, to model various types of multiphase state systems. Because the differential operator in \cref{eq:Allen_Cahn_2D} is linear, similar convergence behavior of the various time-stepping methods as found in \cref{subsec:numerics_SL_advection} may be expected.

The spatial approximation is conducted via spectral methods on a 2D tensor product grid, using $25$ Chebyshev grid points in each dimension. This results in a spatial approximation error of $6.04 \times 10^{-15}$ in the scaled $\ell^2$-norm.

\begin{figure}[ht]
	\begin{minipage}[b]{.46\textwidth}
		\includegraphics[width=\textwidth]{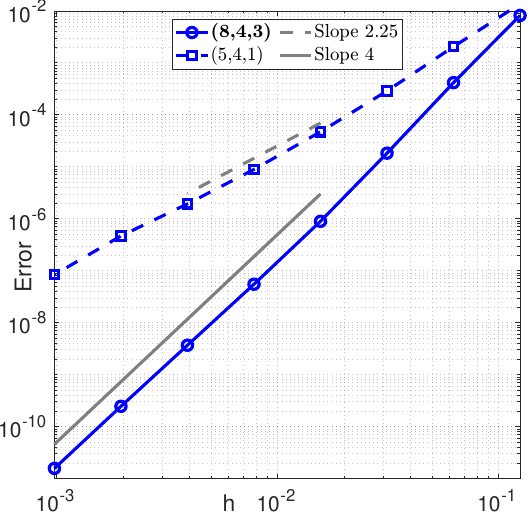}
	\end{minipage}
    \hspace{1em}
	\begin{minipage}[b]{.46\textwidth}
	\includegraphics[width=\textwidth]{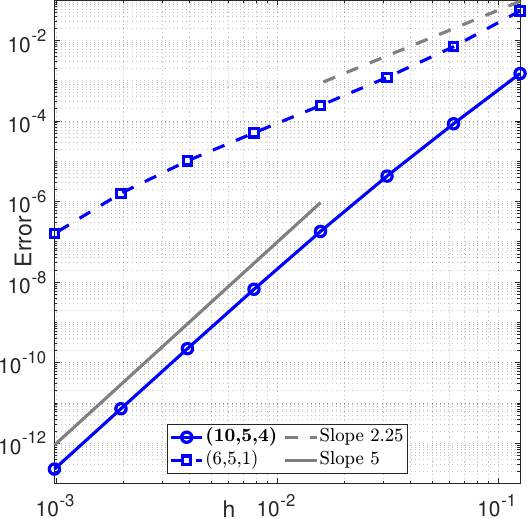}
	\end{minipage}
	\caption{Convergence in the scaled $\ell^2$-norm for the 2D Allen--Cahn equation \cref{eq:Allen_Cahn_2D}. Left panel: fourth order methods with $\psl = 1$ and $\psl = p-1$. Right: fifth order methods with $\psl = 1$ and $\psl = p-1$. The methods exhibit the convergence orders established by the semilinear theory.}
	\label{fig:ErrConvg_AllenCahn2D}
\end{figure}

We now time-step the resulting systems until $t=0.5$ via the various methods, and compute the errors in the scaled $\ell^2$-norm in space.
\Cref{fig:ErrConvg_AllenCahn2D} displays the convergence results of the new methods in comparison against classical reference methods. In line with \cite{rosales2024spatial} and the asymptotic arguments in \cref{subsec:numerics_SL_advection}, the classical methods exhibit a convergence rate of 2.25, despite having formal orders of four and above. The highest spatial derivative is degree two, so numerical boundary layers of thickness $O(h^\frac{1}{2})$ arise in which the error is $O(h^2)$, thus leading to an approximate $L^2$ error of $O((h^\frac{1}{2} \times (h^2)^2)^\frac{1}{2}) = O(h^{2.25})$.

In contrast, the new methods with $\psl = p-1$ exhibit convergence rates (very close to) $p$ for all tested methods, confirming that the conditions and methods presented herein can effectively alleviate order reduction also for problems in higher spatial dimensions.

\subsection{Viscous Burgers' equation}
\label{subsec:viscous_Burgers}
As a stiff nonlinear PDE problem, we study the viscous Burgers' equation,
\begin{equation}\label{eq:Vis_Burgers}
u_t + u u_x = \nu u_{xx}+\psi \ \ \text{for} \ (x,t)\in (0,1) \times (0,1],
\end{equation}
with a viscosity constant $\nu = 0.1$, which is large enough to ensure that the computational grid used accurately captures the spatial scales of the solution, but small enough to ensure that the nonlinear term is relevant for the observed approximation errors.
The initial conditions, the time-dependent Dirichlet boundary conditions, and the source term $\psi(x,t)$ are such that the manufactured solution is $u(x,t) = \cos(2+10t) \sin(0.2+20x)$.
A uniform grid with $\Delta x = 10^{-3}$ is used for the spatial discretization, with both spatial derivatives approximated using sixth-order linear centered finite differences (one-sided near the boundaries). This yields a spatial error of $2.87 \times 10^{-13}$ (found in the same way as in the above tests), thus ensuring that the temporal error is isolated in the convergence study.

\begin{figure}[ht]
	\begin{minipage}[b]{.46\textwidth}
	\includegraphics[width=\textwidth]{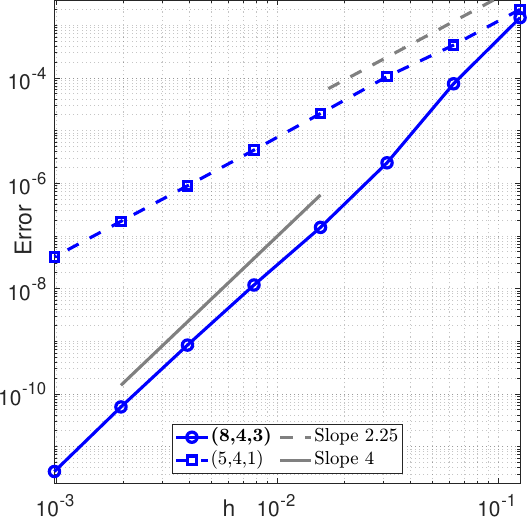}
	\end{minipage}
    \hspace{1em}
	\begin{minipage}[b]{.46\textwidth}
	\includegraphics[width=\textwidth]{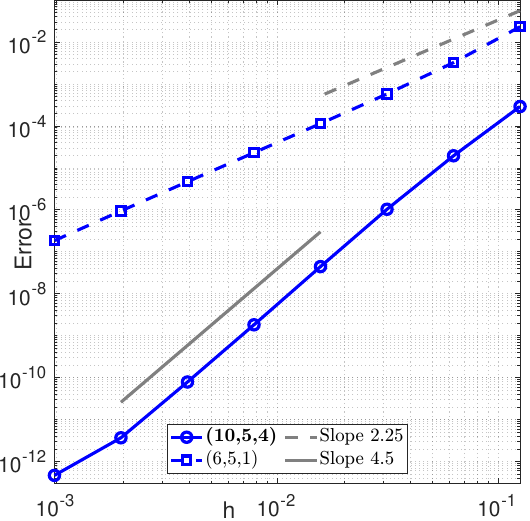}
	\end{minipage}
	\caption{Convergence in the scaled $\ell^2$-norm for the viscous Burgers' equation \cref{eq:Vis_Burgers}. Left panel: fourth order methods with different semilinear orders, in which the new method with high semilinear order exhibits convergence at order $p$. Right panel: fifth order methods with different semilinear orders. Here the new method exhibits convergence order 4.5, which is less than $p$ but significantly better than the order 2.25 of the reference method.}
	\label{fig:ErrConvg_VisBurgers}
\end{figure}

\Cref{fig:ErrConvg_VisBurgers} shows the error convergence results over the range $10^{-3} \le h \le 10^{-1}$, which represents natural step size choices for the given test problem and spatial discretization. Unlike the previous PDE tests, the problem here is not covered by the theory in \cref{Sec:BackgroundSemilinearConditions} because we are not in the regime where the step size $h$ is small relative to the Lipschitz constant of the nonlinear term, which scales like the number of spatial grid points.

The results illustrate that all classical methods converge with an order of $2.25$ in the scaled $\ell^2$-norm (for the same argument as in \cref{subsec:Allen-Cahn}) and hence exhibit order reduction. The new methods, however, converge at orders 4.0 for the (8,4,3) method and 4.5 for the (10,5,4) method. These observations are consistent with related numerical results \cite{biswas2023design}, revealing that weak stage order (in \cite{biswas2023design}) and/or semilinear order conditions can manifest in significant accuracy benefits, even for test cases that lie outside the convergence theory (albeit not necessarily at the full order $p$).

\subsection{Van der Pol equation}
\label{subsec:van_der_Pol}
As an example of a stiff, truly nonlinear ODE not of the semilinear form \cref{eq:ODE}, we consider the Van der Pol test problem as defined in \cite[Chapter~VI.3, p.~406]{hairer1996solving}
\begin{equation} \label{eq:VdP_1}
    \begin{split}
        y'_1 &= y_2,\\
        \epsilon y'_2 &= (1 - y_1^2) y_2 - y_1,
    \end{split}
\end{equation}
with $t \in (0,0.5]$. The initial condition is given by
\begin{equation*}
    (y_1(0), y_2(0)) = \left(2, -\tfrac{2}{3} + \tfrac{10}{81}\epsilon - \tfrac{292}{2187}\epsilon^2 - \tfrac{1814}{19683}\epsilon^3\right).
\end{equation*}
We consider the value $\epsilon = 10^{-3}$, representing a mildly stiff regime, and $\epsilon = 10^{-6}$, corresponding to a strongly stiff regime.

\begin{figure}[htb]
	\begin{minipage}[b]{.46\textwidth}
		\includegraphics[width=\textwidth]{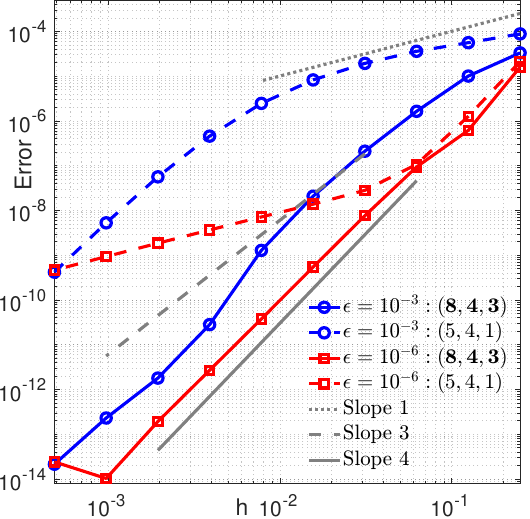}
	\end{minipage}
    \hspace{1em}
 	\begin{minipage}[b]{.46\textwidth}
		\includegraphics[width=\textwidth]{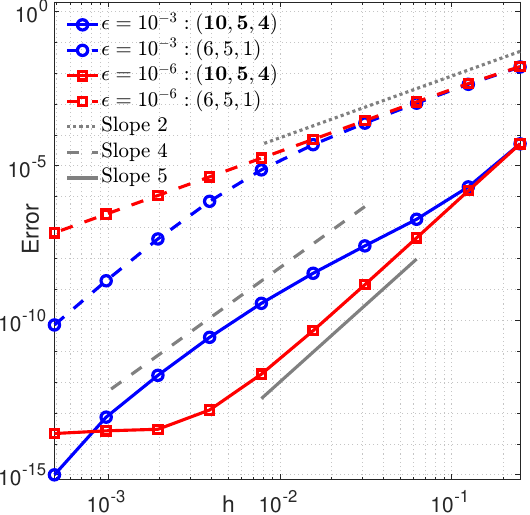}
	\end{minipage}
	\caption{Error convergence for the Van der Pol equation \cref{eq:VdP_1}. Left panel: fourth order methods with different semilinear orders. Right panel: fifth order methods with different semilinear orders. Even though this test lies fully outside the semilinear theory, the new methods with high semilinear order nevertheless alleviate order reduction effects.}
	\label{fig:ErrConvg_VdP_eqn}
\end{figure}

The reference solution at the final time is computed using MATLAB's \verb|ode45| solver with absolute and relative tolerances $5 \times 10^{-14}$. \Cref{fig:ErrConvg_VdP_eqn} shows the numerical convergence rates of solutions obtained through our proposed methods and classical methods. As the problem lies fully outside the class of semilinear problems, our proposed methods are not guaranteed to eliminate order reduction; and the numerical results confirm that in fact, order reduction effects are still present. At the same time, the new methods do consistently enhance the convergence rates relative to the reference schemes, resulting in significantly smaller approximation errors.


\section{Numerical results: Adaptive time-stepping}
\label{sec:adaptive_time_stepping}
In this section we show how important order reduction, and its mitigation, are for error estimation used in adaptive time-stepping. We consider the semilinear advection-reaction test problem \cref{eq:advection-reaction_PDE}, with final time $\tf = 1$, which is a non-forced PDE whose solution transitions through a range of scales of spatio-temporal variability (thus justifying adaptive time-stepping).

The adaptive time-stepping procedure used is as follows. Let $U^n$ denote the numerical solution of the ODE system at time $t_n$. Given a step size $h_n$, the standard adaptive step-size strategy based on local error estimation \cite[Section II.4]{hairer1993solving} computes two approximate solutions, $U^{n+1}$ and $\widehat{U}^{n+1}$, at time $t_{n+1}=t_n+h_n$. The discrepancy between these approximations is used to define a local error indicator
\begin{equation}\label{eq:loc_err_control}
\mathit{err}
=
\left(
\frac{1}{\nvar}
\sum_{i=1}^{\nvar}
\left(
\frac{
U_i^{n+1}-\widehat{U}_i^{n+1}
}{
\mathit{Atol}
+
\max\left\{
\left|U_i^{n+1}\right|,
\left|\widehat{U}_i^{n+1}\right|
\right\}
\mathit{Rtol}
}
\right)^2
\right)^{1/2}.
\end{equation}
A step is accepted when $\mathit{err}\leq1$, in which case the solution is updated using the higher-order approximation $U^{n+1}$, and the step size for the subsequent step is chosen as
\begin{equation}\label{eq:modified_control_step}
h_{\mathrm{new}}
=
\alpha
\left(
\frac{1}{\mathit{err}}
\right)^{1/(\min(p,\widehat{p})+1)}
h_n.
\end{equation}
Here, $\alpha<1$ denotes a safety factor introduced to reduce the likelihood of step rejections, and $\widehat{p}$ denotes the order of the embedded method. Throughout this work, we use $\alpha=0.9$, which is one of the choices suggested in \cite{hairer1993solving}. When $\mathit{err} > 1$, the step is rejected and recomputed using the step size $h_{\mathrm{new}}$ prescribed by \cref{eq:modified_control_step}.
%

We choose a given absolute tolerance $\mathit{Atol}$ and (for simplicity) a relative tolerance $\mathit{Rtol} = 0$. 
Because computational efficiency is a key objective in adaptive time-stepping, we here solve all nonlinear systems up to a Newton tolerance of $10^{-2}\times\mathit{Atol}$. 

For the interpretation of the local errors incurred by the scheme and the embedding when solving an ODE $y' = f(y)$, we define a local reference solution
\begin{equation}
u_{\mathrm{ref}}(t_{n+1}) := y(t_{n+1}),
\end{equation}
where $y$ is the (exact) solution of the ODE with initial value $y(t_n) = U^{n}$. In other words, $u_{\mathrm{ref}}(t_{n+1})$ reveals what a numerical approximation $U^{n}$ at time $t_{n}$ would turn into when advanced without error to time $t_{n+1}$. For the numerical results, $u_{\mathrm{ref}}(t_{n+1})$ is computed almost up to machine precision.




Using the above procedure, we have at each step the scheme's (high order) approximation $U^n$ and its embedding's approximation $\widehat{U}^n$. We therefore have three errors: (i)~the error estimator $\|U^n - \widehat{U}^n\|$, (ii)~the method's local error $\|U^n - u_\mathrm{ref}(t_n)\|$, and (iii)~the embedding's local error $\|\widehat{U}^n - u_\mathrm{ref}(t_n)\|$. An effective adaptive time-stepping routine generally relies on the assumption that the method's error is smaller than the embedding's error, and thus the error estimator can be used as a proxy for the latter.

\begin{figure}[htbp]
    \begin{subfigure}[t]{\textwidth}
        \includegraphics[width=.49\textwidth]{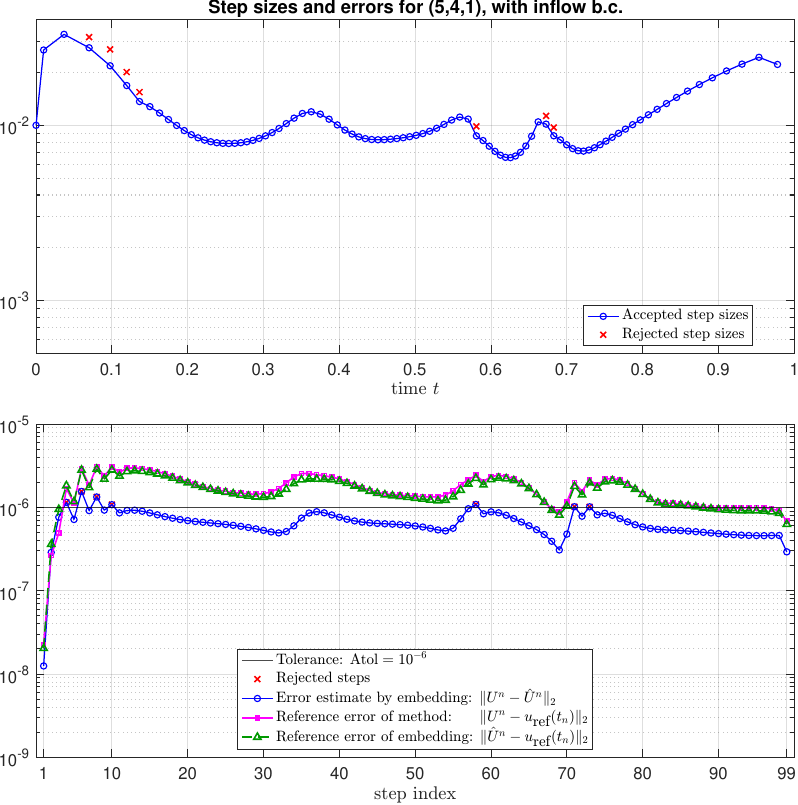}\hfill
        \includegraphics[width=.49\textwidth]{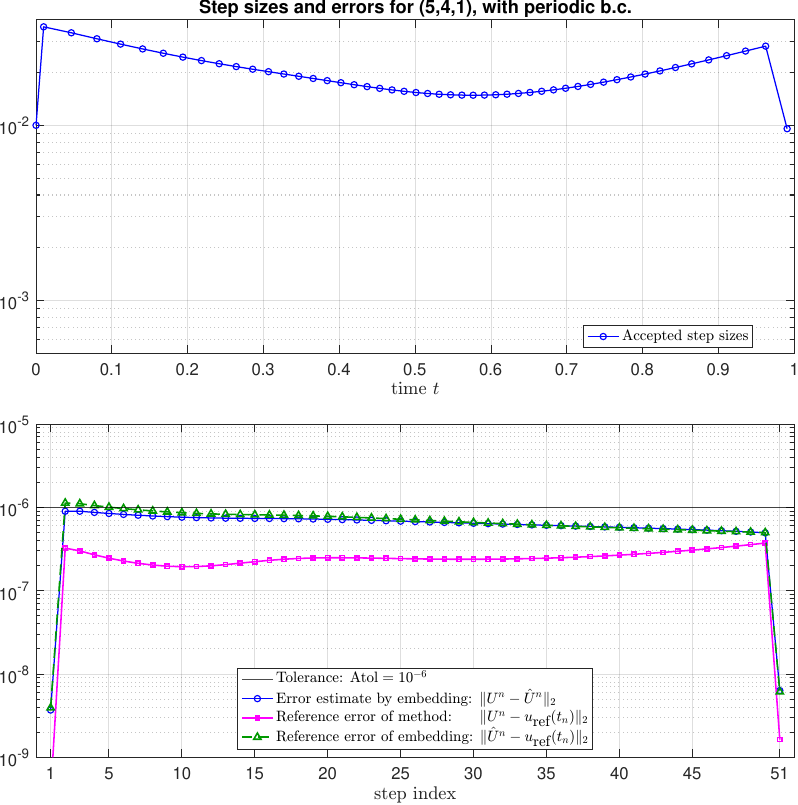}
        \caption{Step size (top row) and local error (bottom row) profiles for the reference (5,4,1) method.
        With periodic boundary conditions (right column), there is no order reduction and the step size appropriately adapts to keep the local error within the tolerance. With inflow boundary conditions (left column), the method experiences order reduction, underestimates the local truncation error, and takes nearly twice as many steps.}
    \end{subfigure}
    
    \vspace{.5em}
    \begin{subfigure}[t]{\textwidth}
        \includegraphics[width=.49\textwidth]{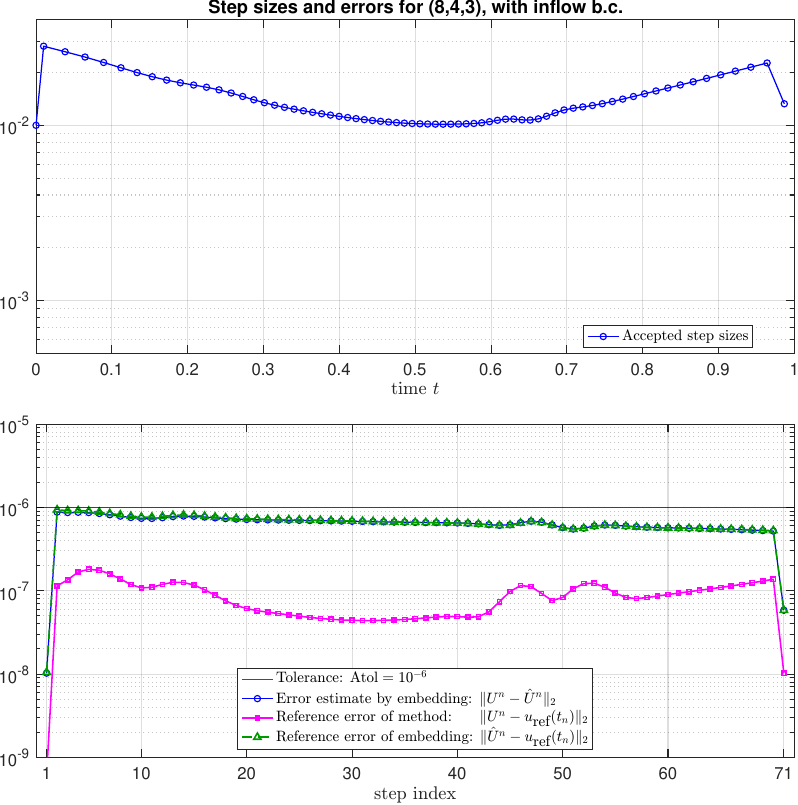}\hfill
        \includegraphics[width=.49\textwidth]{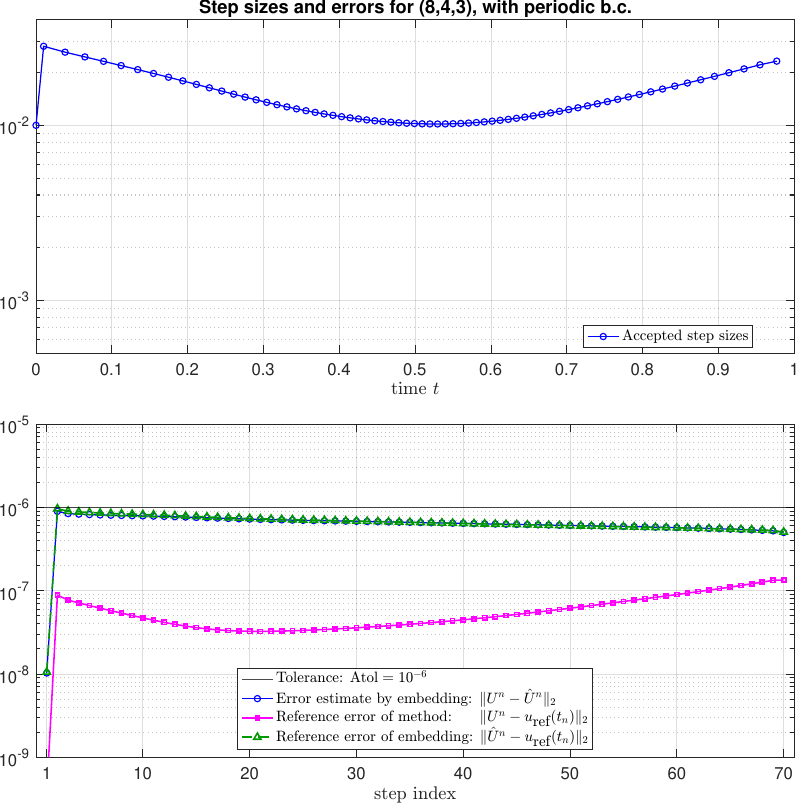}
        \caption{Step size (top row) and local error (bottom row) profiles for the new (8,4,3) method.
        With either inflow (left column) or periodic (right column) boundary conditions, there is no order reduction and the step size appropriately adapts to keep the local error within the tolerance.}
    \end{subfigure}
    \caption{Step size adaptivity comparison of the fourth order methods applied to the semilinear advection-reaction PDE \cref{eq:advection-reaction_PDE}.}
    \label{fig:adaptive_pde_4th_order_stepsizes}
\end{figure}

First we study fourth order methods, considering the new (8,4,3) method together with its embedding, and comparing it against the reference method (5,4,1) with its embedding. We choose $\mathit{Atol} = 10^{-6}$. The numerical results are visualized in the following sequence of figures. For each of the two methods, \cref{fig:adaptive_pde_4th_order_stepsizes} shows the step size $h_n$ vs.\ time $t_n$ and (immediately below it) the relevant errors defined in the previous paragraph. The weighted $\ell^2$-norm is used for all spatial errors.

Because we expect that the reference method may exhibit order reduction near inflow boundaries, we simulate the PDE \cref{eq:advection-reaction_PDE} with the solution \cref{eq:advection-reaction_solution}, once with inflow and then with periodic boundary conditions. The results in \cref{fig:adaptive_pde_4th_order_stepsizes} clearly confirm our expectations: (8,4,3) with both types of boundary conditions and (5,4,1) with periodic boundary conditions exhibit qualitatively similar results with local error magnitudes as expected. In contrast, the (5,4,1) with inflow boundary conditions behaves fundamentally inferior: the error estimator is a poor estimator of the local error (of both the high-order scheme and the embedding). Hence, when step sizes are chosen so that the estimator satisfies the tolerance, the actual errors incurred fail to meet the tolerance. Even then, it incurs more accepted and rejected steps than the other situations. This indicates that the reference method can exhibit a poor error estimation, with concerning effects on adaptive time-stepping, while the new method with high semilinear order performs as desired.


\begin{figure}[htbp]
    \includegraphics[width=.49\textwidth]{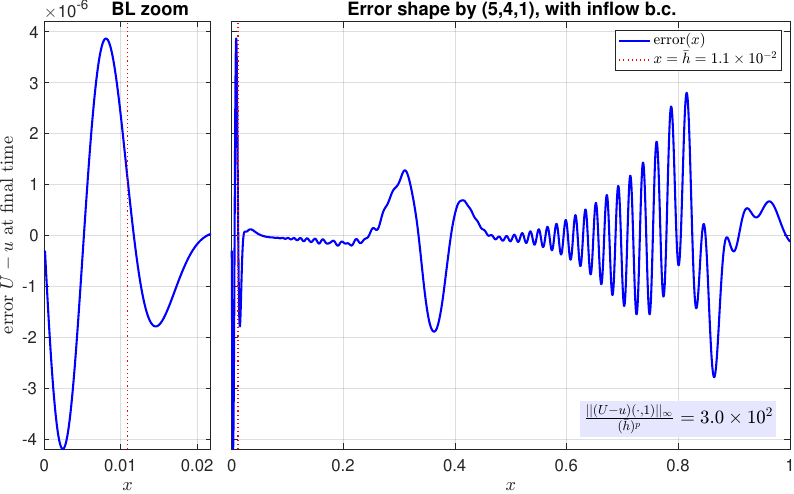}\hfill
    \includegraphics[width=.49\textwidth]{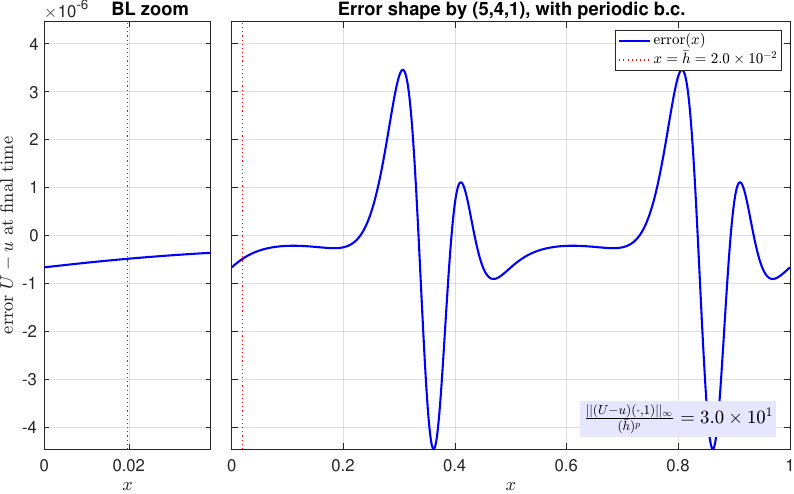}\\[.5em]
    \includegraphics[width=.49\textwidth]{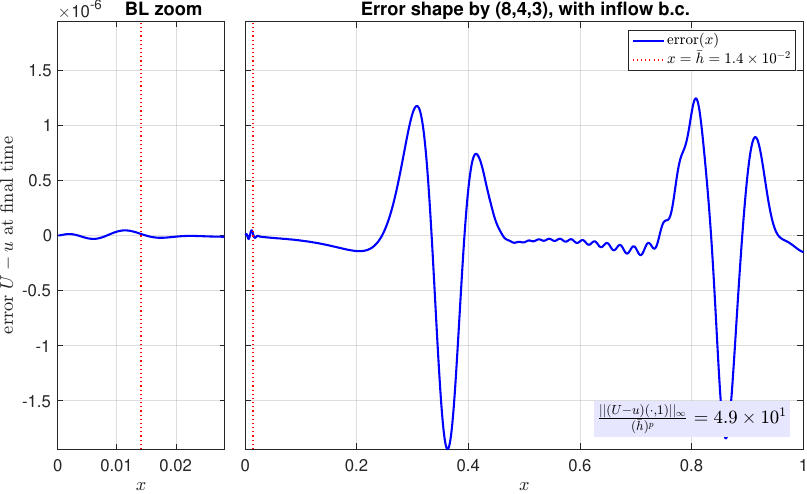}\hfill
    \includegraphics[width=.49\textwidth]{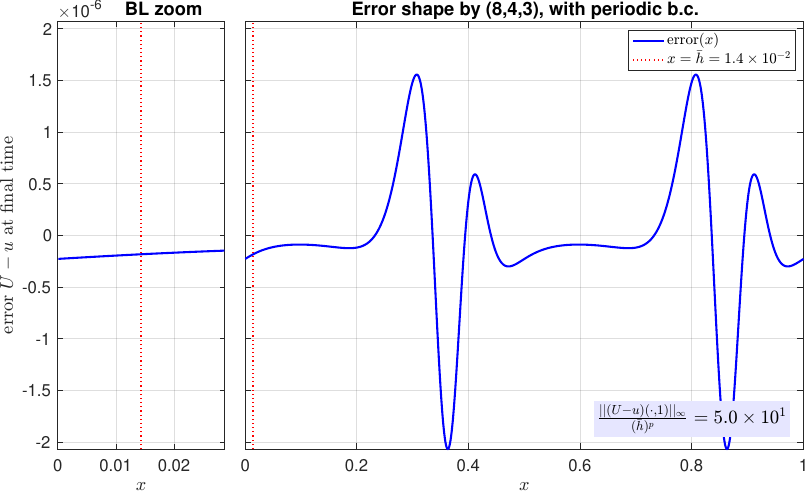}
    \caption{Spatial errors $(U-u)(\cdot,1)$ at the final time for fourth order methods applied to the semilinear advection-reaction PDE \cref{eq:advection-reaction_PDE}, with inflow (left column) and periodic (right column) boundary conditions. Asymptotic analysis \cite{rosales2024spatial} predicts a numerical boundary layer near inflow boundaries; hence, each plot shows a zoom near the left boundary, with the dashed red line at the position of the average step size $\bar{h}$. Indeed, the reference (5,4,1) method (top row) exhibits such an inflow boundary layer (plus, some error artifacts that moved into the domain). In contrast, on a periodic domain such effects are absent. As expected, the (8,4,3) method (bottom row) shows no artifacts in the periodic case either. But importantly, even with inflow boundary conditions, boundary layer effects are insignificant to the error. The error normalized by $\bar{h}^p$
    (light blue box in each panel) also reveals the order reduction effects of (5,4,1) with inflow boundaries.}
    \label{fig:adaptive_pde_4th_order_error_shapes}
\end{figure}

We now provide evidence that indeed an order reduction effect, in the form of a numerical boundary layer (cf.~\cite{rosales2024spatial}), is responsible for the observed behavior. To that end, \cref{fig:adaptive_pde_4th_order_error_shapes} shows the global error at $\tf = 1$ as a function of $x$, for each of the two methods and types of boundary conditions. Each graph also has a zoom (of thickness $O(\bar{h})$ where $\bar{h}$ is the mean step size) near the left boundary. One can clearly see that the reference (5,4,1) method produces a numerical boundary layer near the inflow boundary plus, some artifacts resulting from boundary errors having moved into the domain. In contrast, the absence of inflow boundaries or the use of the high semilinear order (8,4,3) method results in the absence of any boundary layers.

Each panel also shows an error magnitude (measured in the max norm), divided by the mean step size to the power $p$. If the simulations were convergent at the method's order $p$, then the resulting quotient values should be of the same order of magnitude. However, in the case of order reduction, that scaled error is an order of magnitude worse than in situations devoid of order reduction.


These results provide evidence that in the presence of order reduction, the (5,4,1) scheme produces essentially the same error as its embedding (in \cref{fig:adaptive_pde_4th_order_stepsizes}, the magenta and green error curves almost coincide while the blue curve is notably lower). Hence, the resulting numerical solution is not a good approximation of the true solution due to boundary layers---but this fact is invisible to the error estimator because both the method and its embedding produce similar errors. In contrast, when order reduction is avoided, such as via methods satisfying the semilinear order conditions, the error estimators and the associated adaptive time-stepping exhibit the desired behavior.

\begin{figure}[htbp]
    \centering
	\begin{minipage}[b]{.47\textwidth}
		\includegraphics[width=\textwidth]{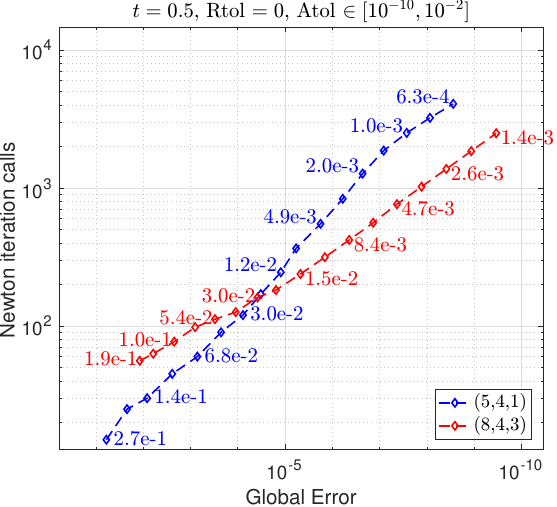}
	\end{minipage}
    \hspace{1em}
	\begin{minipage}[b]{.47\textwidth}
	\includegraphics[width=\textwidth]{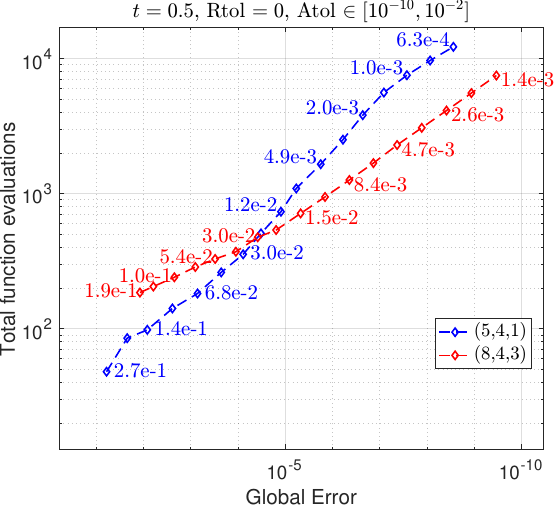}
	\end{minipage}
    \caption{Work-precision diagram for the $(5,4,1)$ and $(8,4,3)$ methods. For a wide range of tolerances, the arithmetic mean step size $\bar{h}$ (text near markers), the global error (horizontal axis), and a measure of total work (vertical axes) are reported. Left measures work via the total number of Newton solves, while right counts the total number of right hand side evaluations. For all data points except for the rightmost three blue points in each graph, the problem is in the stiff regime. While for loose tolerances, the lower number of stages renders the $(5,4,1)$ method more efficient, the increasing superiority of the $(8,4,3)$ method for tighter tolerances is apparent.}
    \label{Fig:SemiLinAdvRe_InflowBCs_DIRK843_WorkPrecision}
\end{figure}

With these structural insights on the manifestations and implications on error estimation of order reduction for initial-boundary-value problems, we now investigate the comparative performance of the new (8,4,3) method vs.\ the reference (5,4,1) method. For the same advection-reaction test problem (here with inflow boundary conditions only), we consider the work-precision diagram incurred by adaptive time-stepping, displayed in \cref{Fig:SemiLinAdvRe_InflowBCs_DIRK843_WorkPrecision}. For a sequence of prescribed tolerances ranging in $\mathit{Atol} \in [10^{-10},10^{-2}]$, each method is run up to $t_\mathrm{f} = 0.5$, and the following information is recorded: the global error and the cost, yielding a point in the log-log diagram, as well as the mean step size $\bar{h}$ used, shown as text next to every other point. In the left panel, cost is measured as the total number of Newton solves (which is equal to the method's number of implicit stages times the number of step sizes, including rejected steps). In the right panel, the total number of evaluations of the ODE's right hand side is counted.


This work-precision comparison is important because it justifies under which circumstances the increased number of stages required to satisfy the semilinear order conditions indeed pays off. The results in \cref{Fig:SemiLinAdvRe_InflowBCs_DIRK843_WorkPrecision} show that the data points for each method lie approximately on lines, but with notably different slopes (note that the rightmost three points of the blue curves exhibit a different slope because there the problem is not stiff anymore, given the choice of $\Delta x$). As a consequence, for low accuracies, the reference method is preferable. However, for higher accuracies, the situation reverts and the new method with high semilinear order is a notably more effective choice.

\begin{figure}[htbp]
    \begin{subfigure}[t]{\textwidth}
        \includegraphics[width=.49\textwidth]{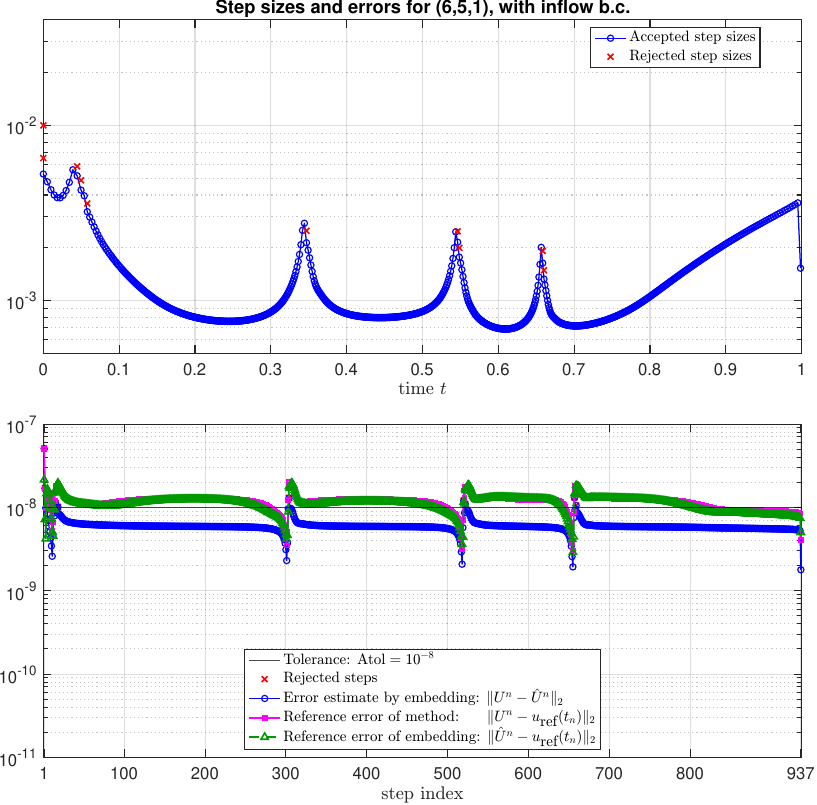}\hfill
        \includegraphics[width=.49\textwidth]{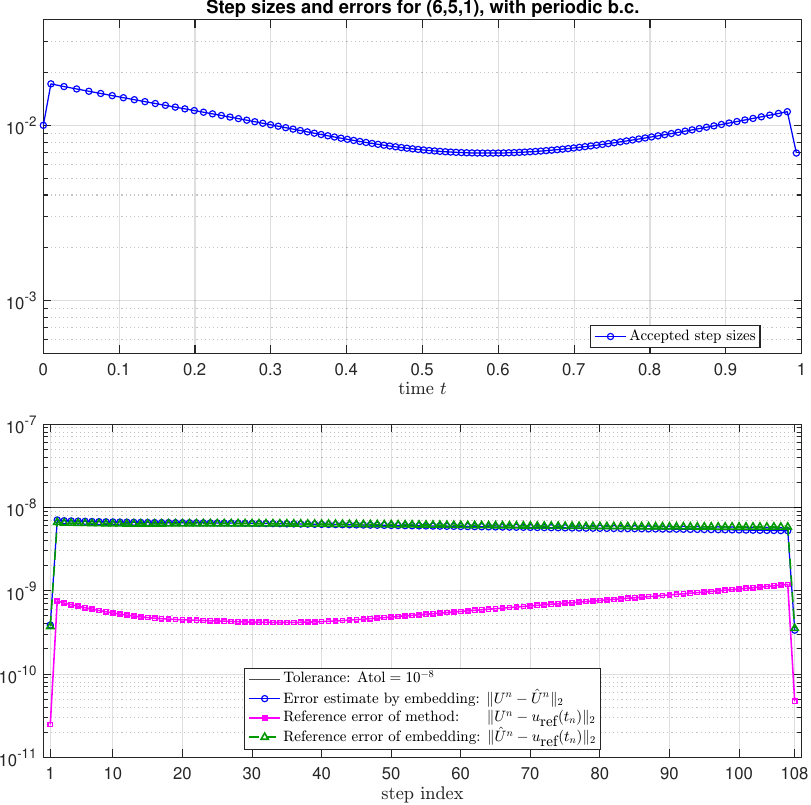}
        \caption{Step size (top row) and local error (bottom row) profiles for the reference (6,5,1) method. With periodic boundary conditions (right column), there is no order reduction and the step size appropriately adapts to keep the local error within the tolerance. With inflow boundary conditions (left column), the method experiences order reduction, underestimates the local truncation error, and takes nearly nine times as many steps.}
    \end{subfigure}
    \vspace{.5em}
    \begin{subfigure}[t]{\textwidth}
        \includegraphics[width=.49\textwidth]{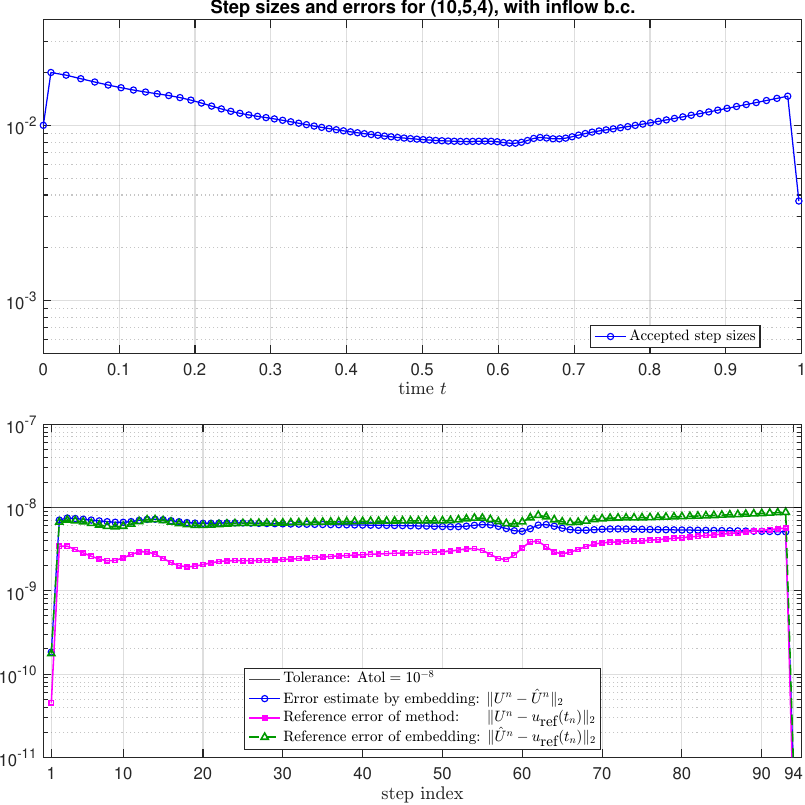}\hfill
        \includegraphics[width=.49\textwidth]{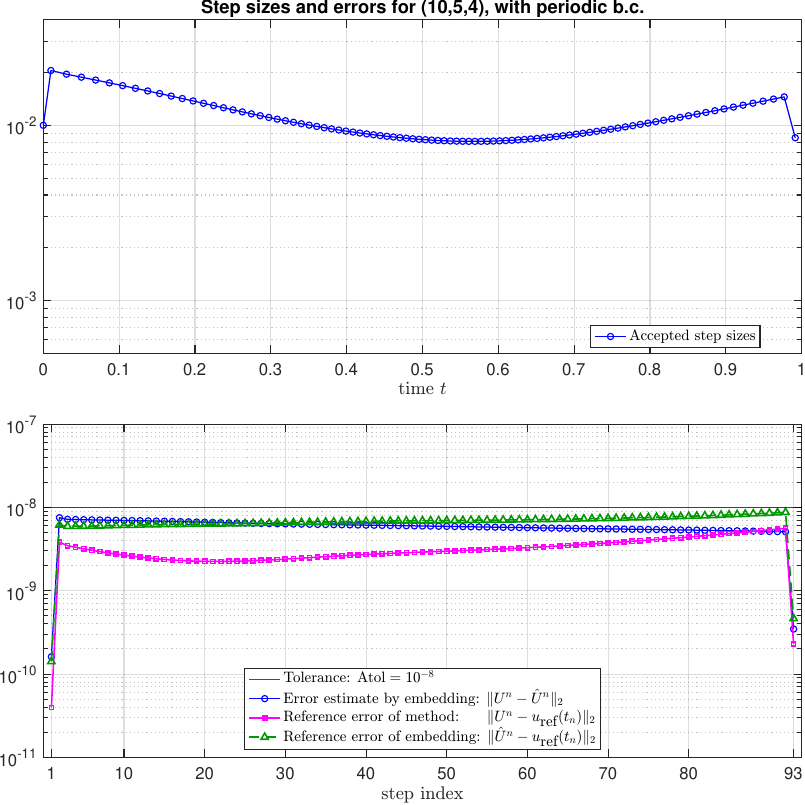}
        \caption{Step size (top row) and local error (bottom row) profiles for the new (10,5,4) method.
        With either inflow (left column) or periodic (right column) boundary conditions, there is no order reduction and the step size appropriately adapts to keep the local error within the tolerance.}
    \end{subfigure}
    \caption{Step size adaptivity comparison of the fifth order methods applied to the semilinear advection-reaction PDE \cref{eq:advection-reaction_PDE}.}
    \label{fig:adaptive_pde_5th_order_stepsizes}
\end{figure}

\begin{figure}[htbp]
    \includegraphics[width=.49\textwidth]{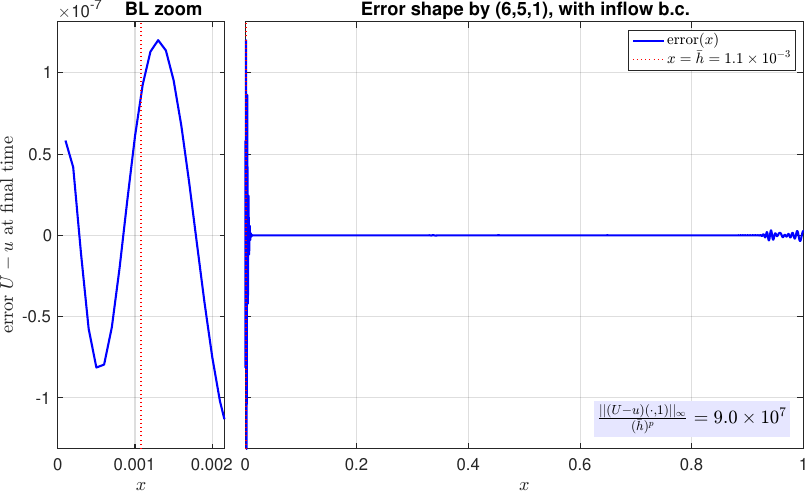}\hfill
    \includegraphics[width=.49\textwidth]{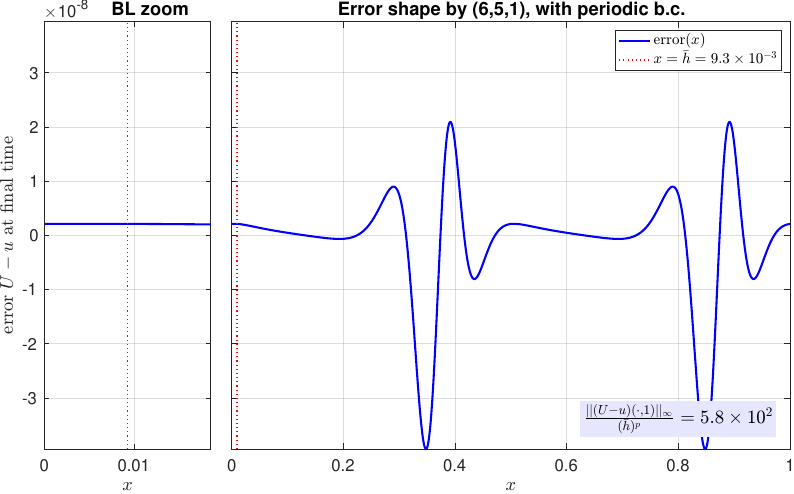}\\[.5em]
    \includegraphics[width=.49\textwidth]{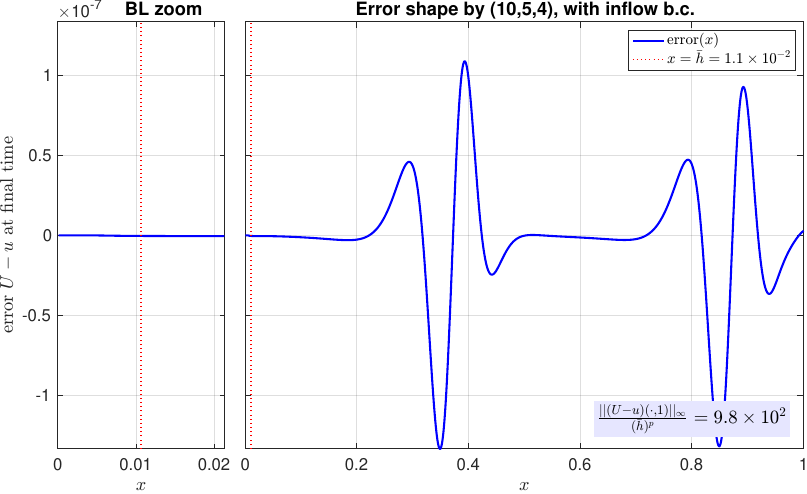}\hfill
    \includegraphics[width=.49\textwidth]{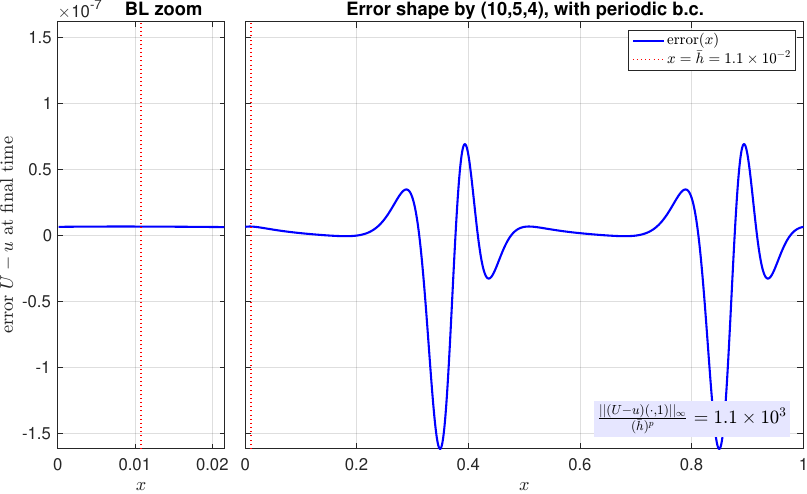}
    \caption{Spatial errors $(U-u)(\cdot,1)$ at final time for fifth order methods, applied to the semilinear advection-reaction PDE \cref{eq:advection-reaction_PDE}, with inflow (left column) and periodic (right column) boundary conditions. The same plotting conventions as in \cref{fig:adaptive_pde_4th_order_error_shapes} apply. The (6,5,1) method exhibits a significant boundary layer near the inflow boundary, resulting in a scaled reference error that is 4 orders of magnitude worse than the other three situations without order reduction.}
    \label{fig:adaptive_pde_5th_order_error_shapes}
\end{figure}


The analogous sequence of results is provided for fifth order methods as well, comparing the new (10,5,4) method with its embedding against the reference (6,5,1) method with its embedding. Here we choose $\mathit{Atol} = 10^{-8}$.

The step size and error plots in \cref{fig:adaptive_pde_5th_order_stepsizes} and the spatial error shapes and scaled errors in \cref{fig:adaptive_pde_5th_order_error_shapes} exhibit the same qualitative behavior as the corresponding results for the fourth order methods (\cref{fig:adaptive_pde_4th_order_stepsizes,fig:adaptive_pde_4th_order_error_shapes}). However, now the impact of order reduction is strikingly more severe: order reduction increases the number of steps by an order of magnitude (in addition to local errors not satisfying the prescribed tolerance); and the boundary layer generates a scaled max error that is 4 orders of magnitude larger than in the absence of order reduction.

\begin{figure}[htbp]
    \centering
	\begin{minipage}[b]{.47\textwidth}
		\includegraphics[width=\textwidth]{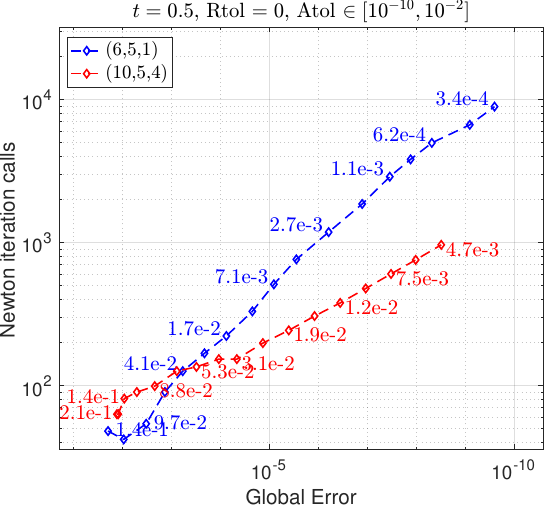}
	\end{minipage}
    \hspace{1em}
	\begin{minipage}[b]{.47\textwidth}
	\includegraphics[width=\textwidth]{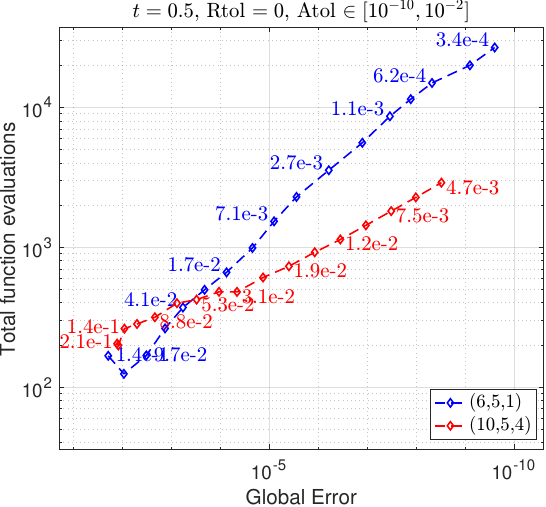}
	\end{minipage}
    \caption{Work-precision diagram for the $(6,5,1)$ and (10,5,4) methods. The plotting style, as well as the key findings, are analogous to those in \cref{Fig:SemiLinAdvRe_InflowBCs_DIRK843_WorkPrecision}. In particular, all points except for the rightmost 3 for the (6,5,1) method are in the stiff regime, and a significant separation in efficiency arises between the two methods.}
	\label{Fig:SemiLinAdvRe_InflowBCs_DIRK1054_WorkPrecision}
\end{figure}

Because additionally the (10,5,4) method incurs slightly fewer step sizes than the (6,5,1) method in the absence of order reduction (periodic boundary conditions), the work-precision graphs in \cref{Fig:SemiLinAdvRe_InflowBCs_DIRK1054_WorkPrecision} identify the high semilinear order method as clearly favorable, except for such loose tolerances that only a dozen step sizes are incurred. This indicates that the additional stages required to satisfy the stiff order conditions more than pay off for problems in which order reduction effects may arise.

\section{Conclusion}
\label{sec:conclusion}
In this work, we have presented the first DIRK methods that attain fourth and fifth order convergence without order reduction for an important class of stiff semilinear ODEs.
To derive these new schemes, we leveraged the order condition theory of \cite{roberts2026runge}, reformulated the conditions into a simpler matrix form, then used a combination of symbolic and numerical optimization techniques to compute coefficients.
For practitioners, we recommend the ESDIRK-(8,4,3) and ESDIRK-(10,5,4) schemes due to their L-stability, optimized principal error, and quality embeddings.
The other two methods presented in this paper, EDIRK-(7,4,4) and EDIRK-(19,5,4), are less optimized and thus are primarily of theoretical interest. Nevertheless, they highlight that methods with stage order one can achieve semilinear order four.

Through extensive numerical experiments, we have shown that our new schemes attain their full and theoretically predicted convergence orders for semilinear problems and even manage to partially mitigate order reduction on nonlinear problems that fall outside the theory.
In contrast, stiffness can degrade the convergence and efficiency of many existing methods from the literature which fail to satisfy the stiff semilinear order conditions.
The effect of order reduction on adaptive step size control is important but scarcely explored.
In this work, we found that methods with low semilinear order can underestimate the local truncation error, overestimate the optimal step size, and exceed user-specified tolerances.

One important focus of future work can be the establishment of $N$-independent (spatial mesh-independent) upper bounds on the error for schemes satisfying semilinear order conditions. When \cref{thm:MainLTE} is applied to a family of ODEs with different $N$, such as a sequence of spatial discretizations of a PDE, the bounds $M$, $L$, and hence the error constant $\overline{D}$ in \cref{thm:MainLTE} may exhibit a dependence on $N$ (primarily because the error bounds are stated in a discrete $L^2$ norm). However, the numerical PDE tests in this work do not indicate a dependence of the overall error on the spatial mesh size.

The implicit Runge--Kutta framework considered in this work generally requires nonlinear solves. Therefore, another important direction of study is to extend the presented framework to other numerical methods that may require less demanding solves, such as Rosenbrock methods or implicit-explicit Runge--Kutta methods.




\section*{Acknowledgements}
This work was performed under the auspices of the U.S.\ Department of Energy by Lawrence Livermore National Laboratory under Contract DE-AC52-07NA27344. LLNL-JRNL-2024060. This material is based upon work supported by the National Science Foundation under Grant No.\ DMS--2309728 (Seibold) and DMS--2309727 (Shirokoff). Any opinions, findings, and conclusions or recommendations expressed in this material are those of the authors and do not necessarily reflect the views of the National Science Foundation.


\appendix

\begin{landscape}
 \section{Table of method properties}

    \begin{table}[ht!]
        \centering
        \small
        \def\arraystretch{1.5}
        \begin{tabular}{r|cccccccccc}
            Method & Source & $A^{(p+1)}$ & $B^{(p+1)}$ & $C^{(p+1)}$ & $D$ & $E^{(p+1)}$ & Stability & $R(\infty)$ & \shortstack{Embedded \\ Stability} & $\widehat{R}(\infty)$ \\
            \hline \hline
            \text{SDIRK-(5,4,1)} & \cite[SDIRK4M]{kennedy2016diagonally} & $2.01 \times 10^{-3}$ & $2.29$ & $1.27$ & $3.03$ & $1.33$ & L & $0$ & A & $-0.5$ \\ \hline
            
            \textbf{ESDIRK-(8,4,3)} & \cref{subsec:symbolic_derivation} & $3.06 \times 10^{-3}$ & $1.50$ & $1.41$ & $1$ & $0.13$ & L & $0$ & L & $0$ \\ \hline
            
            \textbf{EDIRK-(7,4,4)} & \cref{subsec:numeric_derivation} & $1.12 \times 10^{-1}$ & --- & --- & $9.10$ & --- & A & $0.99$ & --- & --- \\ \hline \hline
    
            \text{SDIRK-(6,5,1)} & \cite[Table 2.3]{ismail1998embedded} & $3.97 \times 10^{-3}$ & $1.91$ & $1.61$ & $2.04$ & $1.11$ & A & $-0.40$ & A & $0.54$ \\ \hline
            
            \textbf{ESDIRK-(10,5,4)} & \cref{subsec:symbolic_derivation} & $4.64 \times 10^{-3}$ & $2.03$ & $2.31$ & $1.98$ & $1.10$ & L & $0$ & A & $0.21$ \\ \hline
            
            \textbf{EDIRK-(19,5,4)} & \cref{subsec:numeric_derivation} & $1.12 \times 10^{-2}$ & --- & --- & $9.10$ & --- & A(\ang{89.8}) & $0.98$ & --- & --- \\ \hline
        \end{tabular}
        \caption{Properties of methods presented in this paper (denoted with bold) and comparable methods from the literature.
        The methods column uses the naming convention \textit{TYPE}-($s$, $p$, $\psl$), where $s$ is the number of stages, $p$ is the classical order, and $\psl$ is the semilinear order.
        Next, $A^{(p+1)}$ represents the principal error norm (the 2-norm of the classical order $p+1$ residuals).
        The quality of the embedding is determined by $B^{(p+1)}$, $C^{(p+1)}$, and $E^{(p+1)}$. The maximum coefficient in magnitude is given by $D$.
        Finally, we provide properties of the linear stability region, both for the primary and embedded method.}
        \label{tab:method_properties}
    \end{table}
\end{landscape}

\section{New method coefficients}
In this appendix, we provide coefficients, accurate to 34 digits, for three new methods.
All schemes are diagonally implicit and stiffly accurate so $a_{i,j} = 0$ for $1 \leq i < j \leq s$, and $b_i = a_{s, i}$ for $i = 1, \dots, s$.
The nodes $c_i$ are given by $c_i = \sum_{j=1}^{i} a_{i,j}$.


\subsection{Coefficients for ESDIRK-(8,4,3)}
\label{app:ESDIRK-(843)}
The eight-stage ESDIRK-(8,4,3) method is fourth-order accurate with a semilinear order of three.
The coefficients are as follows:
\begin{equation*}
    \begin{butchertableau}{cl|cl|cl}
        a_{1,1} & 0 & a_{2,1} & \frac{31}{125} & a_{2,2} & \frac{31}{125} \\
        a_{3,1} & \frac{3781}{15500} & a_{3,2} & \frac{63}{124} & a_{3,3} & \frac{31}{125} \\
        a_{4,1} & -\frac{3882222210210885}{75584786387396543} & a_{4,2} & -\frac{3882222210210885}{75584786387396543} & a_{4,3} & 0 \\
        a_{4,4} & \frac{31}{125} & a_{5,1} & \frac{34038088698073943}{122803704925069405} & a_{5,2} & \frac{33514318812866834}{119213963756997001} \\
        a_{5,3} & \frac{224887786579749}{60595115130919582} & a_{5,4} & \frac{4253927007933940}{115874777755193681} & a_{5,5} & \frac{31}{125} \\
        a_{6,1} & \frac{137658149652207956}{209706981851726679} & a_{6,2} & \frac{126492513018975825}{128664872604952432} & a_{6,3} & \frac{11417678526293581}{37223122310090063} \\
        a_{6,4} & -\frac{103353126478507816}{135174297737314759} & a_{6,5} & -\frac{5}{7} & a_{6,6} & \frac{31}{125} \\
        a_{7,1} & \frac{117786594983325079}{151727549241844762} & a_{7,2} & \frac{18526475144695067}{21268081176298953} & a_{7,3} & \frac{7206985701555927}{80976168939093068} \\
        a_{7,4} & -\frac{94099054066115167}{95522373062038575} & a_{7,5} & -\frac{26}{27} & a_{7,6} & \frac{190069087194766309}{197235632620571833} \\
        a_{7,7} & \frac{31}{125} & a_{8,1} & -\frac{29602757552094}{1071399797354437} & a_{8,2} & \frac{548139805377293}{112676962442277364} \\
        a_{8,3} & \frac{424515922983497}{13913811815644881} & a_{8,4} & \frac{67814305287223931}{162780150685834614} & a_{8,5} & -\frac{41854401642916128}{116143966895455495} \\
        a_{8,6} & \frac{74958030483037457}{95092867575812394} & a_{8,7} & -\frac{21188129}{211373000} & a_{8,8} & \frac{31}{125}
    \end{butchertableau}
\end{equation*}
The embedded weights are defined as $\widehat{b}_i = a_{7,i}$ for $i = 1, \dots, 8$.

\subsection{Coefficients for EDIRK-(7,4,4)}
\label{app:EDIRK-(744)}
The seven-stage EDIRK-(7,4,4) method is fourth-order accurate with a semilinear order of four.
The coefficients are as follows:
\begin{equation*}
    \begin{butchertableau}{cl|cl|cl}
        a_{1,1} & 0 & a_{2,1} & \frac{66719178356146069}{104971894986575178} & a_{2,2} & \frac{66719178356146069}{104971894986575178} \\
        a_{3,1} & \frac{11574878994758291}{117719113355115783} & a_{3,2} & -\frac{1858197540898696}{70529361366069153} & a_{3,3} & \frac{11617133062216757}{43245479316548780} \\
        a_{4,1} & \frac{312078294212599530}{40823424700776821} & a_{4,2} & \frac{155312269009595199}{86710391005988198} & a_{4,3} & -\frac{743789150637775609}{113352218631221311} \\
        a_{4,4} & \frac{98271968880200657}{179019545289054999} & a_{5,1} & \frac{1246868775297421168}{137070970121741807} & a_{5,2} & \frac{114921713922407255}{52367417556902641} \\
        a_{5,3} & -\frac{205947502305419261}{24454220481972685} & a_{5,4} & \frac{18936671640200689}{104159855867653343} & a_{5,5} & \frac{47397311839212708}{127463680130367391} \\
        a_{6,1} & -\frac{151740509096074388}{196613682401464609} & a_{6,2} & \frac{254369392774793867}{44087509892864172} & a_{6,3} & -\frac{73864359103986538}{65744654972066205} \\
        a_{6,4} & -\frac{37706375961306427}{179802732674457709} & a_{6,5} & \frac{7953265906419399}{38933344132172515} & a_{6,6} & \frac{48325866641079469}{46020097947328612} \\
        a_{7,1} & \frac{3312403043354842}{33496693975407517} & a_{7,2} & -\frac{7745264544994559}{74509708869668763} & a_{7,3} & \frac{75463258779378077}{134382831179297809} \\
        a_{7,4} & -\frac{11696764876217691}{132584149662964151} & a_{7,5} & \frac{9114026243344448}{106054923174086269} & a_{7,6} & \frac{55946924902076}{75756035623695139} \\
        a_{7,7} & \frac{34834932759942553}{78271243704016222} & \text{} & \text{} & \text{} & \text{}
    \end{butchertableau}
\end{equation*}
This method is primarily of academic interest and therefore does not include an embedding.

\subsection{Coefficients for ESDIRK-(10,5,4)}
\label{app:ESDIRK-(1054)}
The ten-stage ESDIRK-(10,5,4) method is fifth-order accurate with a semilinear order of four.
The coefficients are as follows:
\begin{equation*}
    \begin{butchertableau}{cl|cl|cl}
        a_{1,1} & 0 & a_{2,1} & \frac{23704630662296147}{85251944606883963} & a_{2,2} & \frac{23704630662296147}{85251944606883963} \\
        a_{3,1} & -\frac{3828053375722559}{66474452137650668} & a_{3,2} & -\frac{3828053375722559}{66474452137650668} & a_{3,3} & \frac{23704630662296147}{85251944606883963} \\
        a_{4,1} & -\frac{1768624057837184}{97907440837578655} & a_{4,2} & -\frac{1768624057837184}{97907440837578655} & a_{4,3} & \frac{32342625154963567}{63657230586675651} \\
        a_{4,4} & \frac{23704630662296147}{85251944606883963} & a_{5,1} & \frac{38689688244273643}{145529658745107532} & a_{5,2} & \frac{38689688244273643}{145529658745107532} \\
        a_{5,3} & \frac{5165440406558499}{48318409929685292} & a_{5,4} & 0 & a_{5,5} & \frac{23704630662296147}{85251944606883963} \\
        a_{6,1} & -\frac{2048929420167937}{62953617727398324} & a_{6,2} & -\frac{2048929420167937}{62953617727398324} & a_{6,3} & -\frac{5242126029351595}{74777062100226619} \\
        a_{6,4} & 0 & a_{6,5} & 0 & a_{6,6} & \frac{23704630662296147}{85251944606883963} \\
        a_{7,1} & -\frac{14097385432048041}{83960854026807536} & a_{7,2} & -\frac{14097385432048041}{83960854026807536} & a_{7,3} & \frac{102631915790147645}{94805799854610222} \\
        a_{7,4} & -\frac{5088592904909032}{78095945223394903} & a_{7,5} & \frac{4194495314217601}{126041470655949480} & a_{7,6} & -\frac{25019099907264765}{59359248890957054} \\
        a_{7,7} & \frac{23704630662296147}{85251944606883963} & a_{8,1} & -\frac{6666023823723632}{60782333039950069} & a_{8,2} & -\frac{6666023823723632}{60782333039950069} \\
        a_{8,3} & \frac{34950019030688054}{58700988111542175} & a_{8,4} & \frac{16157137791883982}{129136549150431411} & a_{8,5} & -\frac{4633364877709012}{111280515475374397} \\
        a_{8,6} & -\frac{44830020893037844}{115925512623522465} & a_{8,7} & -\frac{10683392218257989}{83044047426145149} & a_{8,8} & \frac{23704630662296147}{85251944606883963} \\
        a_{9,1} & \frac{7974359957524127}{41777369871990865} & a_{9,2} & \frac{7974359957524127}{41777369871990865} & a_{9,3} & \frac{31122288307425661}{48758046711807126} \\
        a_{9,4} & \frac{124912727797607611}{63031353902083464} & a_{9,5} & -\frac{68608793789563332}{113404568873370149} & a_{9,6} & -\frac{95359235367355842}{59441684582698261} \\
        a_{9,7} & -\frac{162051494287025479}{83831556722521602} & a_{9,8} & \frac{148921337658127561}{79960515909413127} & a_{9,9} & \frac{23704630662296147}{85251944606883963} \\
        a_{10,1} & -\frac{10206283873289495}{82173081853978556} & a_{10,2} & -\frac{10206283873289495}{82173081853978556} & a_{10,3} & 0 \\
        a_{10,4} & -\frac{149692166756442484}{122226586801197919} & a_{10,5} & \frac{59118216399459218}{50501318642781983} & a_{10,6} & \frac{239399454367668061}{196562057586860935} \\
        a_{10,7} & \frac{62112483136249447}{41672907966740429} & a_{10,8} & -\frac{117867017378953048}{115349226975194417} & a_{10,9} & -\frac{40768154109170907}{61569993216212962} \\
        a_{10,10} & \frac{23704630662296147}{85251944606883963} & \hat{b}_1 & -\frac{7377933185266438}{48541560297323275} & \hat{b}_2 & -\frac{7377933185266438}{48541560297323275} \\
        \hat{b}_3 & 0 & \hat{b}_4 & \frac{39662348139301097}{87268009515385808} & \hat{b}_5 & \frac{68262268363872686}{135528001467088899} \\
        \hat{b}_6 & \frac{58007573143164457}{132805403176109435} & \hat{b}_7 & -\frac{28930654619832593}{286820031950471742} & \hat{b}_8 & \frac{15187635870586502}{50762060951677207} \\
        \hat{b}_9 & -\frac{247981587118689503}{472717409267306743} & \hat{b}_{10} & \frac{4}{17} & \text{} & \text{}
    \end{butchertableau}
\end{equation*}

\bibliographystyle{elsarticle-num-names} 
\bibliography{main.bib}

\end{document}